\documentclass{mcom-l}

\usepackage{amssymb}
\numberwithin{equation}{section}

\usepackage{graphicx}

\usepackage{times}
\usepackage{multirow}
\usepackage{enumitem}

\usepackage[normalem]{ulem}
\useunder{\uline}{\ul}{}

\usepackage{fullpage}

\usepackage{fancyhdr}
\usepackage{datetime}
\usepackage{mathtools} 

\usepackage{amsrefs}
\usepackage[colorlinks=true, pdfstartview=FitH, linkcolor=blue,citecolor=blue, urlcolor=blue]{hyperref}
\BibSpec{webpage}{%
  +{}{\PrintAuthors} {author}
  +{,}{ \textit} {title}
  +{}{ \parenthesize} {date}
  +{,}{ \url} {address}
  +{,}{ } {note}
  +{.}{ } {transition}
  +{}{ \parenthesize} {accessdate}
}

\usepackage{keycommand}
    \begingroup
      \makeatletter
      \catcode`\/=8 %
      \@firstofone
        {\endgroup
          \renewcommand{\ifcommandkey}[1]{%
            \csname @\expandafter \expandafter \expandafter 
            \expandafter \expandafter \expandafter  \expandafter
            \kcmd@nbk \commandkey {#1}//{first}{second}//oftwo\endcsname } } 

\newkeycommand\funcU[delta=\delta,m=m,q=q,H=H][1]{
  U_{\commandkey{q},\commandkey{m}} \!\left( #1 ; \commandkey{delta},\commandkey{H} \right) }

\newtheorem{theorem}{Theorem}[section]
\newtheorem{lemma}[theorem]{Lemma}
\newtheorem{prop}[theorem]{Proposition}
\newtheorem{cor}[theorem]{Corollary}
\theoremstyle{definition}
\newtheorem{definition}[theorem]{Definition}
\newtheorem{example}[theorem]{Example}

\newtheorem{conjecture}[theorem]{Conjecture}

\theoremstyle{remark}

\numberwithin{equation}{section}

\renewcommand{\mod}[1]{{\ifmmode\text{\rm\ (mod~$#1$)}\else\discretionary{}{}{\hbox{ }}\rm(mod~$#1$)\fi}}
 \renewcommand{\epsilon}{\varepsilon}

\newcommand{\CC}{{\mathbb C}}

\newcommand{\RR}{\mathbb R}
\newcommand{\Zchi}{{\mathcal Z(\chi)}}
\newcommand{\Zero}{{\mathcal Z}}

\DeclareMathOperator{\Arg}{Arg}

\renewcommand{\phi}{\varphi}

\begin{document}

\title[Bounding $N(T,\chi)$]{Counting Zeros of Dirichlet
  $L$-Functions \\ \today\ \currenttime}
\author{Michael A. Bennett} \address{Department of Mathematics \\
  University of British Columbia \\ Room 121, 1984 Mathematics Road \\
  Vancouver, BC, Canada V6T 1Z2} \email{bennett@math.ubc.ca}
\author{Greg Martin} \address{Department of Mathematics \\ University
  of British Columbia \\ Room 121, 1984 Mathematics Road \\ Vancouver,
  BC, Canada V6T 1Z2} \email{gerg@math.ubc.ca} \author{Kevin O'Bryant}
\address{Department of Mathematics \\ City University of New York,
  College of Staten Island and The Graduate Center \\ 2800 Victory
  Boulevard \\ Staten Island, NY, USA 10314}
\email{kevin.obryant@csi.cuny.edu} \author{Andrew Rechnitzer}
\address{Department of Mathematics \\ University of British Columbia
  \\ Room 121, 1984 Mathematics Road \\ Vancouver, BC, Canada V6T 1Z2}
\email{andrewr@math.ubc.ca} 
\begin{abstract}
  We give explicit upper and lower bounds for $N(T,\chi)$, the number of zeros of a
  Dirichlet $L$-function with character $\chi$ and height at most $T$. Suppose that $\chi$ has conductor $q>1$, and that $T\geq 5/7$.
   If $\ell=\log\frac{q(T+2)}{2\pi}> 1.567$,  then
    \begin{equation*}
      \left| N(T,\chi) - \left( \frac{T}{\pi} \log\frac{qT}{2\pi e} -\frac{\chi(-1)}{4}\right) \right| 
      \le 0.22737 \ell + 2 \log(1+\ell) - 0.5.
    \end{equation*}
We give slightly stronger results for small $q$ and $T$. Along the way, we prove a new bound on $|L(s,\chi)|$ for $\sigma<-1/2$.
\end{abstract}
\maketitle
\tableofcontents

\thispagestyle{empty}

\section{Statement  of  Results}\label{sec-introduction}
For any Dirichlet character $\chi$, the Dirichlet $L$-function is defined  by
\begin{equation} \label{L function def} 
  L(s,\chi) \coloneqq  \sum_{n=1}^{\infty} \frac{\chi (n)}{n^s}
\end{equation}
when $\Re s>1$, and by analytic continuation for other complex
numbers~$s$. We adopt the usual convention of letting $\rho = \beta+i\gamma$ denote a zero of $L(s,\chi)$, so that
$\beta=\Re\rho$ and $\gamma=\Im\rho$ by definition. We let
\begin{equation} \label{Zchi def} 
   \Zchi \coloneqq  \{ \rho \in\CC \colon 0<\beta < 1, \, L(\rho,\chi)=0 \}
\end{equation}
be the set of zeros of $L(s,\chi)$ inside the critical strip (technically a multiset, since multiple zeros, if any, are included
according to their multiplicity). Notice in particular that the set $\Zchi$ does not include any zeros on the imaginary axis, even when
$\chi$ is an imprimitive character; consequently, if $\chi$ is induced by another character $\chi^*$, then $\Zchi=\mathcal{Z}(\chi^*)$. If $\bar\chi$ is the conjugate character to $\chi$, then $\Zchi = \overline{\mathcal{Z}(\bar\chi)}$.

We write $N(T,\chi)$ for the standard counting function for zeros of $L(s,\chi)$ with $0 < \beta < 1$ and $|\gamma| \leq T$. In other words,
$$
N(T,\chi) \coloneqq  \#\{ \rho\in\Zchi \colon |\gamma| \le T\},
$$
counted with multiplicity if there are any multiple zeros. The primary aim of this work is to provide explicit upper and lower bounds on $N(T,\chi)$
in terms of $\chi(-1)$, the conductor $q$ and the height $T$.

\begin{theorem}\label{thm:Main}
  Let $\chi$ be a character with conductor $q > 1$ and let $T\geq 5/7$. Set $\ell\coloneqq \log\frac{q(T+2)}{2\pi}$. If $\ell  \leq 1.567$, then $N(T,\chi)=0$. If $\ell > 1.567$,  then
  \begin{equation*}
    \left| N(T,\chi) - \left( \frac{T}{\pi} \log\frac{qT}{2\pi e} -\frac{\chi(-1)}{4}\right) \right| 
    \le 0.22737 \ell + 2 \log(1+\ell) - 0.5.
  \end{equation*}
\end{theorem}

There have been two earlier papers dedicated to finding explicit bounds for the quantity $N(T,\chi)$, by McCurley~\cite{McCurley} in 1984 and by Trudgian~\cite{Trudgian} in
2015. Both authors gave bounds of the shape
    \begin{equation} \label{C1 C2 def}
    \left| N(T,\chi) - \frac{T}{\pi} \log\frac{qT}{2\pi e}  \right| 
    \le C_1\log qT +C_2
  \end{equation}
for positive constants $C_1$ and~$C_2$.
In McCurley~\cite{McCurley}, which assumes $T\ge1$, these constants $C_1=C_1(\eta)$ and $C_2=C_2(\eta)$ are functions of a parameter $\eta\in(0,1/2]$; for all such values of~$\eta$, one finds that necessarily $C_1(\eta)>{1}/{\pi\log2} > 0.45$.
Trudgian pushed McCurley's techniques further, giving~\cite[Theorem~1]{Trudgian} a table of ten pairs of values $(C_1,C_2)$ under the assumption $T\geq 1$ and ten further pairs under the assumption $T\geq 10$. All
of his pairs have $C_1\geq 0.247$, and in his proof it is asserted that $C_1$ could be made as small as $(\pi\log4)^{-1}\doteq 0.229612$.

Regrettably, Trudgian's paper contains an error that
renders his proof incomplete. In short, the various parameters  introduced in his proofs need to satisfy certain inequalities, and he incorrectly argued that one of
the inequalities was redundant. The same difficulty unfortunately recurs in~\cite{Trudgian0} (where bounds are derived for zeros of the Riemann zeta-function) and in~\cite[Theorem~2]{Trudgian} (devoted to analogous results for Dedekind zeta-functions). On a certain level, the main purpose of the paper at hand is to repair these problems for Dirichlet $L$-functions, motivated by the fact that the authors appealed to~\cite[Theorem~1]{Trudgian} in the course of proving the main results of~\cite{EBPAP}.

Our bound in Theorem~\ref{thm:Main} has a slightly more complicated shape (and uses the offset of $\chi(-1)$) to make the bound as small as possible; however, for any $C_1>0.22737$, it is a simple calculus exercise to calculate a constant~$C_2$ such that Theorem~\ref{thm:Main} implies the bound~\eqref{C1 C2 def}. We can therefore deduce the following corollary of Theorem~\ref{thm:Main} in a straightforward way:
\begin{cor}\label{cor:simple}
  Let $\chi$ be a character with conductor $q>1$. If $T\ge 5/7$, then
    \begin{equation*}
    \left| N(T,\chi) - \frac{T}{\pi} \log\frac{qT}{2\pi e}  \right| 
    \le \min\{0.247\log qT +6.894, 0.298 \log qT+4.358 \}.
  \end{equation*}
\end{cor}
Corollary~\ref{cor:simple} improves upon all twenty of Trudgian's claimed pairs as well as upon McCurley's parametric bound.
Figure~\ref{fig:McCurley and Trudgian} shows the $(C_1,C_2)$ pairs implied by McCurley, and the twenty pairs claimed by Trudgian, as well as the $(C_1,C_2)$ pairs implied for $T\geq 1$ by Theorem~\ref{thm:Main}; the two marked points are the two $(C_1,C_2)$ pairs from Corollary~\ref{cor:simple}.

\begin{figure}
  \includegraphics[width=4.5in]{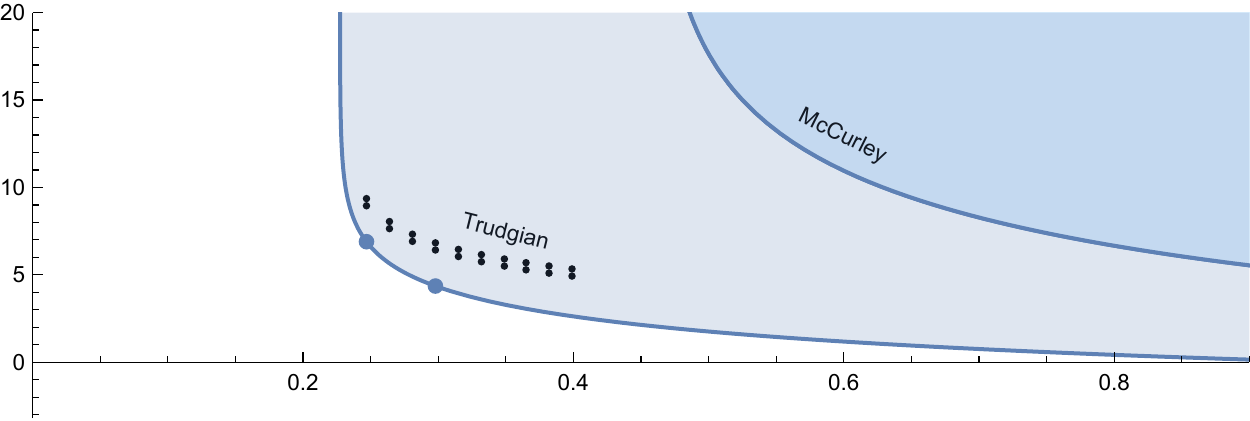}
  \caption{The valid $(C_1,C_2)$ pairs proved by McCurley (upper curve), claimed by Trudgian (twenty points), and implied by Theorem~\ref{thm:Main} (lower curve); the two points on the lower curve represent Corollary~\ref{cor:simple}.}
  \label{fig:McCurley and Trudgian}
\end{figure}

As noted earlier, the current work is focused on fixing the aforementioned  error in \cite{Trudgian}, while at the same time introducing a number of further
improvements. Most notably, in Theorem~\ref{thm:Lboundleft} we extend a bound of
Rademacher~\cite{Rademacher} on $|L(s,\chi)|$ from
$-\frac12 \leq \sigma \leq \frac32$ to all
real~$\sigma$, allowing us to set our parameters more liberally. Also, we
make a choice for~$\eta$ in terms of~$q$ and~$T$ that is nearly optimal,
allowing us to deduce a rather simpler bound. Thirdly, we computed all
806,544 zeros of
primitive $L$-functions, corresponding to 80,818 characters, with $\ell \leq 6$ and $1 < q <  935$,  to
sufficient precision to verify the bounds in
Theorem~\ref{thm:Main} in this range, allowing us to assume greater lower bounds on~$T$ in our proofs. Finally, we are also interested in lower
bounds on $N(T,\chi)$ when $T$ is small, as in Conjecture~\ref{conj:smallT} below, and so we state the
inequality in Theorem~\ref{thm:Main} in a form that is more useful towards that end.

For any fixed $q$, we should note that Theorem \ref{thm:Main} is not particularly of practical interest. The conductor $q$ will either be so large that ``explicit'' is
not helpful, or small enough that one can compute the low height zeros
to great precision. For large $T$, the main term
$\frac{T}{\pi}\log\frac{qT}{2\pi e}$ so greatly exceeds the error
term (even in McCurley's form), that any improvement is truly
minor. Also, the requirement that $T\geq 5/7$ makes our bound
unhelpful for those studying zeros of extremely low height.

Where this result is useful is when $T$ is small, but a large range of values of~$q$ are to be worked with, and the need for an explicit bound arises
not from the large number of zeros that come with large $T$ for one
character but from the large number of characters under consideration. For
example, in~\cite{EBPAP}, the authors needed to treat all moduli up to
$q\leq 10^5$, a total of 1,847,865,075 primitive
characters. McCurley's bound implies that there are at most
32,456,205,589 corresponding zeros of height at most $1$ and conductor at most
$10^5$, while Trudgian's claims (one of which we used in~\cite{EBPAP})
would cut this down to 21,880,443,454. Theorem~\ref{thm:Main}
reduces this number to just 16,461,465,486.  Some computations for low height zeros and the proof of Theorem~\ref{thm:Main}, tailored specifically for $T=1$ and shown in Table~\ref{table:q_a(T,k)} below, lower the number 
still further to just 14,431,705,483.

It is disappointing that, for \emph{fixed} $T$, the main term and the error term in
Theorem~\ref{thm:Main} are of comparable size. We are thus motivated
to conjecture, as we are unable to prove, that the error term should
be an actual  error term, that is, genuinely smaller than the main term. We state this conjecture in a more qualitiative form:

\begin{conjecture}\label{conj:smallT}
  For every real $T>0$ and every integer $M\ge1$, there is an integer~$q_0$ such that
  every character~$\chi$ with conductor at least~$q_0$ satisfies
  $N(T,\chi) \geq M.$
\end{conjecture}
Assuming the generalized Riemann hypothesis for Dirichlet $L$-functions,
Selberg~\cite{Selberg} proved
that the error term in the counting function for $N(T,\chi)$ is
$O\big(\frac{\log q(T+1)}{\log\log q(T+3)}\big)$ uniformly in~$q$
and~$T$; in particular, Conjecture~\ref{conj:smallT} follows from
GRH. McCurley's bound implies that this conjecture holds for
$T> \frac{1}{\log2} \doteq 1.443$, and Theorem~\ref{thm:Main} implies this
conjecture for $T\ge 5/7+
10^{-5}$. By way of example,
we know of characters with conductor $840$ for which $N(1,\chi)=0$;
Theorem~\ref{thm:Main} implies that $N(1,\chi)\geq 1$ when $q\geq 1.3 \times 10^{47}$.
The largest conductor of a character~$\chi$ in our dataset with $N(2,\chi)=0$ is~$241$;
Theorem~\ref{thm:Main} implies that $N(2,\chi)\geq 1$ when $q\geq 1.2\times 10^7$.

Motivated by Selberg's bound and somewhat substantial computation of  zeros, we make a rather speculative conjecture.
\begin{conjecture}
  Let $\chi$ be a character with conductor $q>1$. Recall that $\ell = \log\frac{q(T+2)}{2\pi}$. If $T\ge 5/7$, then
  \begin{equation*}
    \left| N(T,\chi) - \left( \frac{T}{\pi} \log\frac{qT}{2\pi e} -\frac{\chi(-1)}{4}\right) \right| 
    \le \frac{\ell}{\log(2+\ell)}.
  \end{equation*}
\end{conjecture}

The outline of this paper is as follows. In Section \ref{sec-main}, following the approach of McCurley, we derive our first estimates for $N(T,\chi)$, from which our main results will follow. Section \ref{sec-gamma} is devoted to sharp inequalities for the Gamma function. In Section \ref{sec-bounds}, we begin the task of bounding the argument of $L(s,\chi)$, by constructing a function whose zeros measure changes in the argument. 
In Section \ref{sec-jensen}, we complete this process through application of Backlund's trick and Jensen's formula.
Finally, in
Section \ref{sec-assembling}, we complete the proof of Theorem \ref{thm:Main}.

The technical details of our computations can be found in data files accessible at:
\begin{center}
\texttt{\href{http://www.nt.math.ubc.ca/BeMaObRe2/}{\url{http://www.nt.math.ubc.ca/BeMaObRe2/}}}
\end{center}
\section{The Main Term} \label{sec-main}

Assuming that $\chi$ is a primitive character with conductor $q>1$,
the completed $L$-function, an entire function, is defined as
\[\Lambda(s,\chi) \coloneqq  \left(\frac q \pi\right)^{s/2}
  \Gamma\left(\frac{s+a_\chi}{2}\right)L(s,\chi);\]  we note that
the zeros of $\Lambda(s,\chi)$ are precisely those of $L(s,\chi)$. The
{\it functional equation} is
\begin{equation}\label{Functional Equation} \Lambda(s,\chi) =
  \epsilon(\chi) \Lambda(1-s,\bar\chi),\end{equation}
where $\epsilon(\chi)$ is independent of $s$ and has absolute value 1.

Fix $\sigma_1>1$. By integrating $\frac{\Lambda'}{\Lambda}(s,\chi)$
around the rectangle with corners at $\sigma_1\pm iT$ and
$1-\sigma_1 \pm iT$ (where $T$ is not the height of a zero of
$L(s,\chi)$), and appealing to equation~\eqref{Functional Equation} on the left half
of the contour, we arrive at the identity
\begin{equation}\label{eq:von Mangold}
  N(T,\chi) = \frac{T}{\pi} \log\frac{q}{\pi} + \frac{2}{\pi} \Im \ln\Gamma(\tfrac 14 + \tfrac {a_\chi}2+ i\, \tfrac T2) + \frac 1\pi \arg L(s,\chi)\Big|_{s=1/2-iT}^{1/2+i T},
\end{equation}
where
\[a_\chi \coloneqq  \begin{cases} 0, & \text{if } \chi(-1)=1, \\ 1, & \text{if } \chi(-1) =    -1,\end{cases}\]
is the sign of the character.
Define
\begin{equation} \label{GAT}
g(a,T) \coloneqq  \frac2\pi \Im\ln\Gamma(\tfrac14+\tfrac a2+ i \tfrac  T2)-\frac{T}{\pi}\log\frac{T}{2e}-\frac{2a-1}{4},
\end{equation} 
so that 
\begin{equation}\label{eq:main-terms}
  \frac{T}{\pi} \log\frac{q}{\pi} + \frac{2}{\pi} \Im \ln\Gamma(\tfrac 14 + \tfrac {a_\chi}2+ i\, \tfrac T2) 
  =\frac{T}{\pi} \log\frac{qT}{2\pi  e} -\frac{\chi(-1)}{4} + g(a_\chi,T).
\end{equation}
We have that
\begin{align*}
  \arg L(s,\chi)\Big|_{s=1/2-iT}^{1/2+i T} &=
                                             \arg L(s,\chi)\Big|_{s=1/2-iT}^{\sigma_1-i T}
                                             + \arg L(s,\chi)\Big|_{s=\sigma_1-iT}^{\sigma_1+i T}  +\arg L(s,\chi)\Big|_{s=\sigma_1+iT}^{1/2+i T} \\
                                           &=\arg L(\sigma-iT,\chi)\Big|_{\sigma=1/2}^{\sigma_1}
                                             +\arg L(\sigma_1+it,\chi)\Big|_{t=-T}^{T} +\arg L(\sigma+iT,\chi)\Big|_{\sigma=\sigma_1}^{1/2} .
\end{align*}
In particular,
\begin{equation} \label{eq:3 pieces}
 \left| \arg L(s,\chi)\Big|_{s=1/2-iT}^{1/2+i T} \right| \leq \left|
    \arg L(\sigma-iT,\chi)\Big|_{\sigma=1/2}^{\sigma_1} \right|
  +\left| \arg L(\sigma_1+it,\chi)\Big|_{t=-T}^{T} \right|\\
  +\left| \arg L(\sigma+iT,\chi)\Big|_{\sigma=\sigma_1}^{1/2} \right|.
  \end{equation}
These three terms sometimes all have the same
sign in practice, suggesting that there is no possibility of finding cancellation
in general. Since 
$$
\overline{L(\sigma-iT,\chi)}=L(\sigma+iT,\bar\chi),
$$
we have
\[ \left| \arg L(\sigma-iT,\chi)\Big|_{\sigma=1/2}^{\sigma_1} \right|
  \leq \max_{\tau \in \{\chi,\bar\chi\}}\left| \arg L(\sigma+iT,\tau)\Big|_{\sigma=\sigma_1}^{1/2}
  \right| .\] If $\chi$ is a real character then we have equality in this statement, so
again there is no recoverable loss in general.

Trivial bounds on $|L(s,\tau)|$ come from comparing the Euler products of $L(s,\tau)$ and $\zeta(s)$, leading immediately to the following.
\begin{lemma}\label{lem:EPB}
  If $s=\sigma +it$ and $\sigma>1$, then
  \(\displaystyle \frac{\zeta(2\sigma)}{\zeta(\sigma)} \leq  |L(s,\tau)| \leq \zeta(\sigma).  \)
\end{lemma}

\begin{prop}\label{C2 prop}
  For $\sigma_1>1$,
  \(\displaystyle \left| \arg L(\sigma_1+it,\tau)\Big|_{t=-T}^{T}  \right| \le 2 \log\zeta(\sigma_1).\)
\end{prop}

\begin{proof}
  For $t$ between $-T$ and $T$, Lemma~\ref{lem:EPB} implies that
  \[| \arg L(\sigma_1+t i,\tau)| \leq |\log L(\sigma_1+t i,\tau)| \leq    \log \zeta(\sigma_1).\]
  The proposition thus follows from the fact that $\arg L$ is trapped between $-\log\zeta(\sigma_1)$ and $\log\zeta(\sigma_1)$, whereby its net change is at most $2\log\zeta(\sigma_1)$.
\end{proof}

We have thus arrived at the inequality
\begin{equation}\label{eq:One arg term left}
  \left| N(T,\chi) - \left( \frac{T}{\pi} \log\frac{qT}{2\pi e} -\frac{\chi(-1)}{4}+ g(a_\chi,T)\right) \right| \le \frac2\pi \log \zeta(\sigma_1)+ \frac 2\pi \max_{\tau \in \{\chi,\bar\chi\}}\left| \arg L(\sigma+iT,\tau)\Big|_{\sigma=\sigma_1}^{1/2} \right|.
\end{equation}
We will use Stirling's approximation to estimate $g(a_\chi,T)$ in the next section, and the remainder of the work is spent on bounding the change of  $\arg L(\sigma+iT,\tau)$ on the segment $\sigma\in[1/2,\sigma_1]$. Up to this point, we have followed McCurley's approach to the problem verbatim.

\section{The Gamma Function} \label{sec-gamma}

We require bounds for the Gamma function in two contexts. The first of these is  in equation
\eqref{eq:von Mangold} where the real part of the argument  is either $1/4$ or $3/4$, while the second is in the situation where we have a  fixed imaginary
part $T/2$ and varying real part. Both usages are nicely handled by a
suitable shifted version of Stirling's approximation.

\begin{lemma}[Stirling's approximation]\label{lem:Stirling}
  Let $x$ and $y$ be positive real numbers. Then $\Im \ln\Gamma(x+iy)$ is
  within 
  $$
  \frac{(4+3\pi)/1440}{((x+2)^2+y^2)^{3/2}}
  $$
   of the expression
     \begin{multline*}
    y \log \frac{y}{e} + \frac{\pi}{2} \left(x-\frac{1}{2}\right)
    -\left(x+\frac{3}{2}\right) \arctan\frac{x+2}{y}
    -\frac{y/12}{(x+2)^2+y^2} \\
    +\frac{y}{2} \log \left(1+\frac{(x+2)^2}{y^2}\right)
    +\arctan\frac{x}{y} +\arctan \frac{x+1}{y}.
  \end{multline*}
\end{lemma}

\begin{proof}
  By \cite{Hare}*{Proposition 2.1}, we have the identity 
  \begin{equation*}
    \ln\Gamma(z)=\ln\Gamma(z+1)-\log z=\ln\Gamma(z+2)-\log z - \log (z+1).
  \end{equation*}
  Thus,
  \begin{equation}\label{eq:gammashift}
    \Im \ln\Gamma(z)=\Im \ln\Gamma(z+2)-\Arg(z)-\Arg(z+1).
  \end{equation}
  We will use the version of Stirling's series and corresponding error bounds given
  in~\cite{Brent}: for $\Re(z) > 0$, there is a complex function $R_2$
  with $|R_2(z)| \leq \frac{4+3\pi}{1440 |z|^3}$ and
  \[\ln\Gamma(z) =\left(z-\frac12 \right) \log z-z+\frac12 \log 2\pi +
    \frac{1}{12z} + R_2(z).\] For $x$ and $y$ positive real numbers, we have
  \(\Arg(x+i y)=\frac{\pi}{2}-\arctan(x/y).\) Equation ~\eqref{eq:gammashift}
  now becomes
  \begin{align*}
    \Im\ln\Gamma(x+iy) &=
                         y \log \frac{y}{e} + \frac{\pi}{2} \left(x-\frac{1}{2}\right) 
                         -\left(x+\frac{3}{2}\right) \arctan\frac{x+2}{y} \\
                       &\qquad -\frac{y/12}{(x+2)^2+y^2}+\frac{y}{2} \log \left(1+\frac{(x+2)^2}{y^2}\right) \\
                       &\qquad +\arctan\frac{x}{y}  +\arctan \frac{x+1}{y}
                         + \Im R_2(x+2+iy),
  \end{align*}
  and the lemma follows from $|\Im R_2(z+2)|\leq |R_2(z+2)|$.
\end{proof}

\begin{prop}\label{gamma prop}
  For $a\in\{0,1\}$, $T \ge 5/7$ and $g(a,T)$ defined as in (\ref{GAT}), we have
  \[|  g(a,T) |\le \frac{2-a}{50T}.\]
\end{prop}

\begin{proof}
  We need only apply Lemma~\ref{lem:Stirling} with
  $x=\frac a2+\frac 14$ and $y=T/2$, finding that $g(a,T)$ is within
  $\frac{(8+6\pi)/45}{(81+40a+4T^2)^{3/2}}$ of
  \begin{multline}\label{eq:gaT estimate}
    \frac{16+12\pi -60 T \sqrt{40 a+4 T^2+81}}{45 \pi  \left(81+40 a+4 T^2\right)^{3/2}}+\frac{T}{2\pi} \log \left(1+\frac{40 a+81}{4 T^2}\right)  \\
    +\frac{2}{\pi} \left(\arctan\frac{2 a+1}{2 T}+\arctan\frac{2 a+5}{2 T}
      -\left(\frac{a}{2}+\frac{7}{4}\right) \arctan\frac{2 a+9}{2
        T}\right).
  \end{multline}
  Proving the four inequalities (upper and lower, $a=0$ and $a=1$) is a typical problem for interval analysis.
\end{proof}

A number of times in this work we will assert that some inequality is true ``by interval analysis". Full details are available in Mathematica notebooks on the website
\begin{center}
\texttt{\href{http://www.nt.math.ubc.ca/BeMaObRe2/}{\url{http://www.nt.math.ubc.ca/BeMaObRe2/}}}
\end{center}
but we wish to indicate the idea behind this under-utilized technique here. One extends the domain of some primitive real functions (like addition, multiplication, arctangents, logarithms, etc.) to include intervals, and so that
  \[f(X_1,\dots,X_n) = \{ f(x_1,\dots,x_n) \colon x_i \in X_i\}.\]
The fundamental theorem of interval analysis says that if $h$ is defined by a composition of primitive functions and $x_i \in X_i$, then
  \[h(x_1,\dots,x_n) \in h(X_1,\dots,X_n).\]
For instance,~\eqref{eq:gaT estimate}, multiplied by $T$ and with $a=0$ and $T=[1,\frac{129}{128}]$ becomes
  \begin{multline*}
  \bigg[ \frac{12 \pi +16-\frac{5805 \sqrt{38713}}{2048}}{3825 \sqrt{85} \pi }+\frac{\log \left(\frac{38713}{1849}\right)}{2 \pi }+\frac{2 \arctan\frac{25026}{8321}-\frac{7}{2} \arctan\frac{9}{2}}{\pi }  ,\\
  \frac{129}{128} \left(\frac{262144 \left(12 \pi +16-60 \sqrt{85}\right)}{47036295 \sqrt{38713} \pi }+\frac{129 \log \frac{85}{4}}{256 \pi }+2-\frac{2\arctan 12 +\frac{7}{2} \arctan \frac{192}{43}}{\pi }\right)\bigg],
  \end{multline*}
a subset of $[0.022,0.035]$. Also, $T \cdot \frac{(8+6\pi)/45}{(81+40a+4T^2)^{3/2}}$ becomes
  \[
  \bigg[ \frac{262144 (8+6 \pi )}{47036295 \sqrt{38713}},\frac{43 (8+6 \pi )}{163200 \sqrt{85}} \bigg] \subseteq [ 0.0007,0.0008].\]
This computation then constitutes a proof that $0.0213 \leq Tg(0,T)\leq 0.0358$ for $1\leq T \leq \frac{129}{128}$. It should be noted that this proof works without floating point arithmetic, except at moments when one needs to decide which of two expressions represents a smaller number.

One can then proceed to a proof for all $T\geq 5/7$ by breaking the interval $[5/7,\infty)$ into sufficiently small intervals. By the definition of uniform continuity, if the domain is broken into sufficiently small pieces, then interval arithmetic will yield a sufficiently tight bound on the range of the function. There is a theoretical and a practical difficulty with this paradigm for generating proofs of inequalities. The theoretical problem is that we need not only the function to be uniformly continuous, but for every sub-computation involved to be uniformly continuous. This may require cleverly rewriting the expression or by introducing more primitive functions, each such introduction requiring some (usually easy) calculus proof. 

The practical difficulty that arises is that ``sufficiently small pieces'' can quickly become too numerous to be useful. This can be partially addressed by rewriting the expression, but also by introducing a simple expression between the target function and the planned bound. For example,
  \[\frac{T}{2\pi} \log \left(1+\frac{81}{4 T^2}\right)  \leq \frac{T}{2\pi} \times \frac{81}{4 T^2} = \frac{81}{8\pi T},\]
and for $T=[100,200]$ this improves the naive interval arithmetic upper bound of $\frac{100}\pi \log \left(\frac{40081}{40000}\right) \approx 0.064$ to $\frac{81}{800\pi}\approx 0.032$. That  is, a theoretically tighter bound in real arithmetic may be theoretically worse in interval arithmetic. The best expression to use may even depend on the specific interval under consideration.

In the course of our interval analysis bounds in this paper, we use Alirezaei's uncommonly sharp bounds for $\arctan x $~\cite{Alirezaei} and Tops\o e's Pad\'e-inspired bounds for $\log(1+x)$~\cite{Topsoe}.

\begin{definition}\label{def:E}
  For $a\in\{0,1\}$, $d\geq 0$ and $T\geq 5/7$, we define
  \[
    {\mathcal E}(a,d,T) \coloneqq  \left| \Im\ln\Gamma(\tfrac{\sigma+a+iT}{2})
      \Big|_{\sigma=1/2}^{1/2+d} + \Im
      \ln\Gamma(\tfrac{\sigma+a+iT}{2}) \Big|_{\sigma=1/2}^{1/2-d}
    \right|.
  \]
  We set $E(a,d,T)$ to be the expression given in Figure~\ref{fig:def of E}, so that 
  Lemma~\ref{lem:Stirling} applied to the definition of ${\mathcal
    E}(a,d,T)$ gives
  \[ {\mathcal E}(a,d,T) \leq E(a,d,T) \] for $0\leq d < \frac 92$ and $T>0$. While $E$ contains many terms, they are each easy to
  work with computationally. Figure~\ref{fig:Epic} shows $E$ for typical arguments.
\end{definition}

	\begin{figure}
		\begin{align*}
		E(a,d,T) \coloneqq &
		\frac{2 T/3}{(2 a+2 d+17)^2+4 T^2}+\frac{2 T/3}{(2 a-2 d+17)^2+4 T^2}\\
		&-\frac{4 T/3}{(2 a+17)^2+4 T^2} +\frac{T}{2}  \log \left(1+\frac{(2 a+17)^2}{4 T^2}\right) \\	
		&-\frac{T}{4}  \log \left(1+\frac{(2 a+2 d+17)^2}{4 T^2}\right)
		-\frac{T}{4}  \log \left(1+\frac{(2 a-2 d+17)^2}{4 T^2}\right)\\
		&+\frac{(8+6 \pi)/45 }{\left((2 a+2 d+17)^2+4 T^2\right)^{3/2}}	\\
		&+\frac{(8+6 \pi)/45 }{\left((2 a-2 d+17)^2+4 T^2\right)^{3/2}}
		+\frac{2 (8+6 \pi )/45}{\left((2 a+17)^2+4 T^2\right)^{3/2}}\\	
		&+2 \arctan \frac{2 a+1}{2 T}-\arctan \frac{2 a+2 d+1}{2 T}-\arctan \frac{2 a-2 d+1}{2 T}\\
		&+2 \arctan \frac{2 a+5}{2 T}-\arctan \frac{2 a+2 d+5}{2 T}-\arctan \frac{2 a-2 d+5}{2 T}\\
		&+2 \arctan \frac{2 a+9}{2 T}-\arctan \frac{2 a+2 d+9}{2 T}-\arctan \frac{2 a-2 d+9}{2 T}\\
		&+2 \arctan \frac{2 a+13}{2 T}-\arctan \frac{2 a+2 d+13}{2 T}-\arctan \frac{2 a-2 d+13}{2 T}\\
		&+\frac{2 a+2 d+15}{4}  \arctan \frac{2 a+2 d+17}{2 T}\\
		&+\frac{2 a-2 d+15}{4}  \arctan \frac{2 a-2 d+17}{2 T}
		-\frac{2 a+15}{2}  \arctan \frac{2 a+17}{2 T}
		\end{align*}
		\caption{Definition of $E(a,d,T)$, for $a\in \{0,1\}$,    $0\leq d < 9/2$, $T\ge 5/7$.}\label{fig:def of E}
	\end{figure}
	\begin{figure}
		\includegraphics[width=2in]{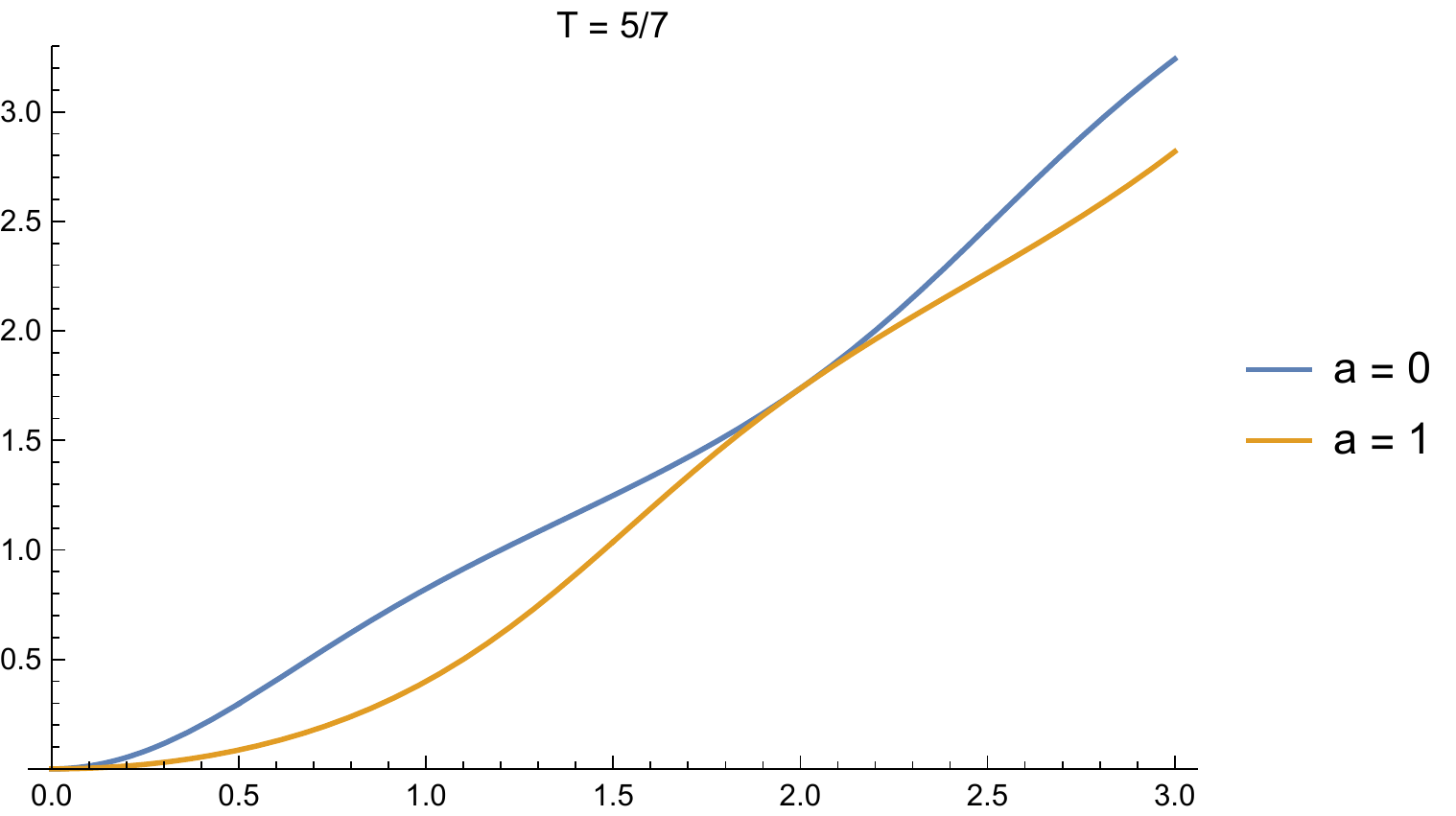}
		\includegraphics[width=2in]{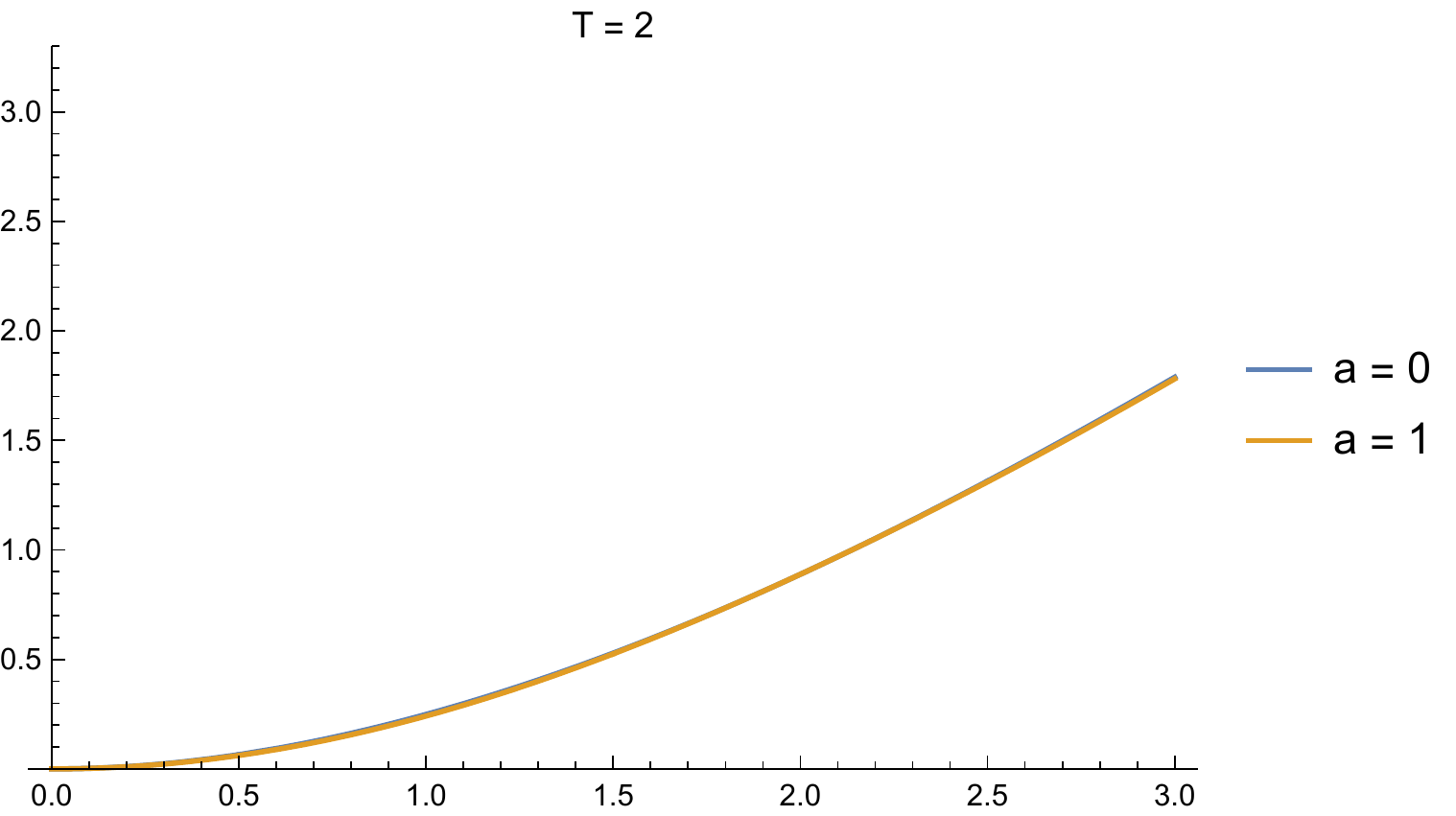}
		\caption{Graphs of ${\mathcal E}(a,d,T)$ for $T=5/7$ and
			$T=2$.}\label{fig:Epic}
	\end{figure}
	
\begin{lemma}\label{lem:E bound}
	Suppose that $0 \le \delta_1 \leq d<\frac 92$, with  $a\in\{0,1\}$ and $T\geq \frac 57$. Then
	\[0< E(a,\delta_1,T) \leq E(a,d,T).\]
	For $a\in\{0,1\}$, $\frac14 \leq d \leq \frac 58$ and $T\ge \frac 57$, 
	\[\frac{E(a,d,T)}{\pi} \leq \frac{(640+216a)d-112-39a}{1536(3T+3a-1)}+\frac{1}{2^{10}}.\]
\end{lemma}

\begin{proof}
  This is proved using interval analysis. For fixed $a$ and $T$, a degree 3 Taylor model with center $0$ is used to show that the derivative of $T\cdot  E(a,d,T)$ with respect to $d$ is positive for small $d$  (using some algebra and the Moore-Skelboe algorithm to bound the 4th  derivative of $E$ with respect to $d$), and the Moore-Skelboe algorithm for larger $d$. As $E(a,0,T)>0$, this shows that $E$ is positive. Consult the website for details.
\end{proof}

\section{Bounds on $L(s,\tau)$, and Bounds on $f_m(s)$} \label{sec-bounds}

\subsection{Introduction of the auxiliary function $f_m$.}
We will construct a function that has many zeros if $\arg L$ changes substantially on the interval on $[\frac12+iT,\sigma_1+iT]$.  To wit, let $m$ be a large
integer, and define, for $s$ a complex number, 
\[f_m(s) \coloneqq  \frac1{2} \big( L(s+iT,\tau)^m +
  L(s-iT,\bar\tau)^m\big).\] 
  Notice the use of $s \pm iT$, rather than
$\sigma\pm iT$; this is done so that $f_m$ is holomorphic.  Note further that, for real
arguments, $f_m$ simplifies nicely:
\begin{align*}
  f_m(\sigma) &=  \frac1{2}\big( L(\sigma+Ti,\tau)^m + L(\sigma-Ti,\bar\tau)^m\big) \\
              & =\frac1{2}\big( L(\sigma+Ti,\tau)^m + \overline{ L(\sigma+Ti,\tau)^m}\big) \\
              &=\Re L(\sigma+Ti,\tau)^m.
\end{align*}
The tactic we will employ follows McCurley \cite{McCurley}. If $\arg L$ changes, then $\arg L^m$ changes $m$
times more, and this causes $L(\sigma+Ti,\tau)^m$ to be purely
imaginary many times, whereby $f_m$ will have many real zeros. We use
Jensen's formula to bound the number of zeros in terms of an integral
of $\log |f_m|$, and then bound the integral using a variety of estimates, trivial and
non-trivial. As $m\to\infty$, this tactic captures the total
variation of $\arg L$, which is sometimes as small as the net change.

\begin{definition}\label{def:n}
  We define $n_m$ to be that integer (depending on $m$) for which
  \begin{equation}\label{line:define n}
    n_m \le \frac 1 \pi \left| \arg L(\sigma+iT,\tau)^m\Big|_{\sigma=\sigma_1}^{1/2}\right| < n_m+1.
  \end{equation}
  Since  we will take $m\to \infty$, the reader will be well-served to
  think of $n_m$ as being very large. We will find upper and lower
  bounds for $n_m/m$.
\end{definition}
  
\begin{lemma}\label{lem:n/m meaning}
  The function $f_m$ has as at least $n_m$ zeros on the real segment
  $[\tfrac12,\sigma_1]$, and
  \[ \frac {n_m}m \leq \frac 1\pi \left| \arg
      L(\sigma+iT,\tau)\Big|_{\sigma=\sigma_1}^{1/2} \right| <
    \frac{n_m+1}m.\]
\end{lemma}


\begin{proof}
  By Definition~\ref{def:n}, we know that the expression
  $\frac12+ \tfrac1\pi \arg L(\sigma+iT,\tau)^m$ is an integer for at least
  $n_m$ different values of $\sigma$ in the interval
  $[1/2,\sigma_1]$. In other words,
  $$
  f_m(\sigma) = \Re L(\sigma+iT,\tau)^m=0
  $$
  for at least $n_m$
  different values of $\sigma$.

  As $m$ is an integer, we have that $\arg L(s,\tau)^m = m \arg
  L(s,\tau)$. Thus, dividing the inequalities in
  line~\eqref{line:define n} by $m$ leads to the desired conclusion.
  \end{proof}

If we had defined $f_m$ with subtraction instead of addition, thereby
picking out the imaginary part instead of the real part of
$L(s,\tau)$, the analogue of the proof of Lemma~\ref{lem:center term} would not be valid.

\begin{lemma}\label{lem:center term}
  For any real $c>1$, there is an infinite sequence of integers $m$
  with $f_m(c) \neq 0$, and moreover, along that sequence
  \[\lim_{m} \bigg( {-}\frac 1m \log |f_m(c)| \bigg) \leq \log
    \frac{\zeta(c)}{\zeta(2c)}.\]
\end{lemma}

\begin{proof}
  Define $K$ and $\psi$ by $L(c+T i,\tau) = Ke^{\psi i}$. Since
  $L(s,\tau)\neq 0$ for $\sigma> 1$ and $c>1$, we know that $K>0$, and
  also $L(c-Ti,\bar\tau)=\overline{L(c+iT,\tau)}=Ke^{-\psi i}$. We
  have
  \[
    \frac{f_m(c)}{L(c+Ti)^m} = \frac 1{2} \left(1 +
      \frac{L(c-Ti,\bar\tau)^m}{L(c+Ti,\tau)^m}\right) = \frac 1{2}
    \left(1 + e^{-2m\psi i} \right).
  \]
  Whatever the value of $\psi$, there is a sequence of values of~$m$ with the property that
  $-2m\psi \to 0 \mod {2\pi}$, whence
  $\frac{f_m(c)}{L(c+Ti)^m} \to 1$.

  For $\sigma>1$, we have the Euler product bound from
  Lemma~\ref{lem:EPB}
  \[\frac{\zeta(2\sigma)}{\zeta(\sigma)} \leq |L(\sigma+it,\tau)|.\]
  This translates, for $m$ a sequence of integers with the above
  property,  into the bound
  \[
    1 =\lim_m \frac{f_m(c)}{L(c+Ti)^m} \leq \lim_m
    \left|\frac{f_m(c)}{\big(\zeta(2c)/\zeta(c) \big)^m} \right|, \]
  which becomes
  \[ 0 \leq \lim_m \log |f_m(c)| - m \log
    \frac{\zeta(2c)}{\zeta(c)},\] 
completing the proof.
\end{proof}

\section{Jensen's formula} \label{sec-jensen}

We apply Jensen's formula to the sequence of functions
$f_m(s)$ and the open disk $D(c,r)$ with center~$c$ and radius~$r$. Here, $m$ ranges
through the sequence of positive integers defined in Lemma~\ref{lem:center term}. As
$\tau$ and $\bar\tau$ are nonprincipal, each $f_m$ is entire, and in
particular holomorphic on $D(c,r)$. Let $\Zero_m(X)$ be the multiset
of zeros of $f_m$ in the set $X\subseteq \CC$.  Let
\begin{equation}\label{def:S}
S_m(c,r) \coloneqq \frac 1m \sum_{z\in \Zero_{m}(D(c,r))} \log \frac r{|z-c|}.
\end{equation}
In our setting and notation, Jensen's
formula is as follows.

\begin{theorem}[Jensen's formula] Let $c\in\CC$, and let $r>0$ be
  real. If $f_m(c)\neq 0$, then
  $$S_m(c,r) = -\frac 1m \log|f_m(c)| + \frac{1}{2\pi} \int_{-\pi}^{\pi} \frac1m \log| f_m(c+r e^{i\theta})|\,d\theta.$$
\end{theorem}
\noindent We apply this to derive an upper bound upon $S_m(c,r)$.

\begin{prop}\label{prop:Jensen}
  Let $c,r$ and $\sigma_1$ be real numbers with
  \[c-r < \frac12 <1 < c < \sigma_1 < c+r,\] and
  $F_{c,r}:[-\pi,\pi]\to\RR$ an even function with
  $F_{c,r}(\theta) \geq \frac1m \log |f_m(c+re^{i\theta})| $.  Then
  \[
    \lim_{m\to\infty} S_m(c,r)
    \leq \log \frac{\zeta(c)}{\zeta(2c)} + \frac1\pi \int_0^\pi F_{c,r}(\theta)\,d\theta.
  \]
\end{prop}
  
We will give a lower bound on the sum that involves $\frac{n_m}{m}$ in
Section~\ref{sec:Backlund}, and an upper bound on the integral, via an explicit $F_{c,r}$, using
classical and new bounds on $L$-functions in Section~\ref{Jennie}. What will then remain is the work of choosing good values for
$c$, $r$ and $\sigma_1$, which we do in Section~\ref{sec-assembling}.

\subsection{Backlund's trick and the Jensen sum}\label{sec:Backlund}

\begin{lemma}\label{lem:using E}
  Let $ d$ and $T$ be positive real numbers, and ${\mathcal E}(a,d,T)$ be
  as in Definition~\ref{def:E}. Then
  \[\left|\arg L(\sigma+iT,\tau)^m \Big|_{\sigma=1/2}^{1/2+ d}\right|
    < \left| \arg L(\sigma+iT,\tau)^m \Big|_{\sigma=1/2}^{1/2-
        d}\right| + m{\mathcal E}(a_\tau, d,T).\]
\end{lemma}

\begin{proof}
  By the functional equation~\eqref{Functional Equation},
  \[ \arg\Lambda(\sigma+iT,\tau) \Big|_{\sigma=1/2}^{1/2+ d} +
    \arg\Lambda(\sigma+iT,\tau) \Big|_{\sigma=1/2}^{1/2- d} = 0.\] 
    Since
  \begin{align*}
    \arg\Lambda(\sigma+iT,\tau)
    &=\arg \left(\frac q \pi \right)^{(s+a_\tau)/2} +\arg \Gamma(\tfrac{s+a_\tau}{2})+ \arg L(s,\tau) \\
    &=\frac t2 \log \frac{q}{\pi} + \Im \ln\Gamma(\tfrac{s+a_\tau}{2})+ \arg L(s,\tau),
  \end{align*}
  we see that the terms
  \[\arg L(\sigma+iT,\tau) \Big|_{\sigma=1/2}^{1/2+ d}
    +\arg L(\sigma+iT,\tau) \Big|_{\sigma=1/2}^{1/2- d} \] and
  \[\Im \ln\Gamma(\tfrac{\sigma+a_\tau+iT}{2})
    \Big|_{\sigma=1/2}^{1/2+ d} + \Im
    \ln\Gamma(\tfrac{\sigma+a_\tau+iT}{2}) \Big|_{\sigma=1/2}^{1/2-
      d}\] add to 0, and so have the same absolute value. This last
  displayed equation has the same absolute value as
  ${\mathcal E}(a_\tau, d,T)$. As
  $\arg L(\sigma+i T,\tau)^m = m \arg L(\sigma+iT)$, we have
  established this lemma.
\end{proof}

We will appeal to the following proposition with rather weak constraints on $c$ and $r$; if $r$ is
much larger than $c$, then we can in fact do slightly better. The source of the error
in~\cite{Trudgian} is in not tracking the constraints on~$c$ and~$r$ and how
they impact the applicability of ``Backlund's trick''.

\begin{prop}[Backlund's trick]\label{prop:Backlund}
  Let $c$ and $r$ be real numbers , and set
  $$
  \sigma_1\coloneqq c+\frac{(c-1/2)^2}{r} \; \mbox{ and } \; \delta\coloneqq 2c-\sigma_1-\frac12.
  $$
  Further, let $E_\delta \coloneqq  E(a_\tau ,\delta ,T)$. If $1<c<r$ and  $0 \leq \delta < \frac 92$, then
  \begin{equation*}
    \left| \arg L(\sigma+iT,\tau) \Big|_{\sigma=\sigma_1}^{1/2}\right|
    \leq \frac{\pi\, S_m(c,r)}{2\log r/(c-1/2) } + \frac{E_\delta}2 +\frac{\pi}{m}.
  \end{equation*}
\end{prop}
\begin{proof}
  The conditions on $c$ and $r$ imply the inequalities 
  $$
  c-r < \frac12-\delta \leq \frac12 \leq
    \frac12+\delta=2c-\sigma_1\le c \le \sigma_1 < c+r.
  $$

  For $z\in D(c,r)$, we see that $\log \frac r{|z-c|} >0$, so that
  \[S_m(c,r) \coloneqq \frac1m \sum_{z\in \Zero_{m}(D(c,r))} \log \frac r{|z-c|} \geq \frac 1m
    \sum_{z\in \Zero_{m}((c-r,\sigma_1])} \log \frac r{|z-c|}.\] We
  will further only consider particular zeros in the real interval
  $(c-r,\sigma_1]$, noting that omitting zeros from the  computation
  weakens rather than invalidates the claimed bound.

  For real $\sigma$, such as those in the interval $(c-r,c+r)$, we have that 
  $$
  f_m(\sigma)=\Re L(\sigma+iT,\tau)^m, 
  $$
  and so for
  $\sigma \in \Zero_{m}((c-r,c+r))$, we have
  \[0=f_m(\sigma ) = \Re L(\sigma + iT, \tau)^m,\] whence
  $\arg f_m(\sigma) = \frac{\pi}{2} + j\pi$ for some integer $j$. By
  the definition of $n_m$, we are then guaranteed at least $n_m$
  values of $\sigma$ in the interval $[1/2,\sigma_1]$ with
  $f_m(\sigma)=0$.

  For $1\leq k \leq n_m$, let $\delta_k$ be the smallest nonnegative
  real number with
  \[f_m(1/2+\delta_k)=0 \qquad\text{and}\qquad k \leq \frac 1\pi
    \left| \arg L(\sigma+iT,\tau)^m \Big|_{\sigma=1/2}^{1/2+\delta_k}
    \right|.\] We set $z_k\coloneqq \frac12+\delta_k$. Define $x_1$ to be the
  number of $z_k$'s that lie in the interval
  $[1/2,1/2+\delta)=[1/2,2c-\sigma_1)$, and let $x_2=n_m-x_1$ be the
  number of $z_k$'s in $[2c-\sigma_1,\sigma_1]$. We have
  \[0\leq \delta_1 <\delta_2<\cdots < \delta_{x_1} <\delta \leq
    \delta_{x_1+1} <\cdots < \delta_{n_m}\leq \sigma_1-1/2.\] Using
  Lemma~\ref{lem:using E},
  \begin{align*}
    k &\leq \frac{1}{\pi} \left| \arg L(\sigma+iT,\tau)^m \Big|_{\sigma=1/2}^{1/2+\delta_k}\right| \\
      &< \frac{1}{\pi} \left| \arg L(\sigma+iT,\tau)^m \Big|_{\sigma=1/2}^{1/2-\delta_k}\right|+ m \, {\mathcal E}(a_\tau,\delta_k,T) \\
      &\leq \frac{1}{\pi} \left| \arg L(\sigma+iT,\tau)^m \Big|_{\sigma=1/2}^{1/2-\delta_k}\right|+ m \,
        E(a_\tau,\delta_k,T).
  \end{align*}
  For each $j\geq 1$, if $k$ is minimal with
  \[\frac 1\pi \left| \arg L(\sigma+iT,\tau)^m
      \Big|_{\sigma=1/2}^{1/2-\delta_k}\right| > k-m \,
    E(a_\tau,\delta_k,T) \geq j,\] then $f_m$ has at least $j$ zeros
  in $[\frac12-\delta_k,\frac12)$. We define $\delta_{-k}$ so that
  $\frac12-\delta_{-k}$ is the largest of the ``at least $j$''
  zeros. We say that the zero $z_k=\frac12+\delta_k$ has a pair,
  namely $z_{-k}=\frac12-\delta_{-k}$. By construction,
  $\delta_{-k}\leq \delta_k$.

  If $z_k\in [\frac12,\frac12+\delta]$ is unpaired, then it
  contributes (using $\frac12+\delta\leq c$)
  \[\frac1m \log\frac{r}{|c-z_k|} = \frac1m
    \log\frac{r}{c-(\frac12+\delta_k)} \geq \frac1m
    \log\frac{r}{c-1/2}\] to $S_m(c,r)$. If $z_k\in [\frac12,\frac12+\delta]$
  is paired, then it (together with its paired zero, which is at least
  $\frac12-\delta$ and so in $(c-r,c+r)$) contributes
  \begin{align*}\frac1m \log\frac{r}{|c-z_k|} + \frac1m
    \log\frac{r}{|c-z_{-k}|}
    &= \frac 1m \log\frac{r^2}{|c-(1/2+\delta_k)| \cdot |c-(1/2-\delta_{-k})|} \\
    &\geq \frac 1m \log\frac{r^2}{|(c-1/2)^2-\delta_k^2|} \geq \frac 1m
    \log\frac{r^2}{(c-1/2)^2}
\end{align*}
to $S_m(c,r)$.  If
  $z_k\in [\frac12+\delta,\sigma_1]$, then it contributes
  \begin{align*}
    \frac1m\log\frac{r}{|c-z_k|}
    &\geq \min\left\{ \frac1m \log\frac{r}{c-(\frac12+\delta)}, \frac1m \log\frac{r}{\sigma_1-c}\right\} \\
    &=\frac1m \log\frac{r}{\max\{c-\frac12-\delta,\sigma_1-c\} }\\
    &=\frac1m\log\frac {r}{\sigma_1-c} = \frac1m \log \frac{r^2}{(c-1/2)^2}
  \end{align*}
  to $S_m(c,r)$, revealing the wisdom in setting $\delta=2c-\sigma_1-\frac12$
  and $\sigma_1=c+\frac{(c-1/2)^2}{r}$.

  Suppose there are $x$ zeros in $[\frac12,\frac12+\delta]$, and $x'$
  of them are unpaired, and there are $n_m-x$ zeros in
  $(\frac12+\delta,\sigma_1]$. We then have
  \begin{align*}
    S_m(c,r) &\geq \frac{x'}m \log \left( \frac{r}{c-1/2} \right)+ \frac {x-x'}{m} \log \left( \frac{r^2}{(c-1/2)^2} \right)\\
    	&\qquad +\frac{n_m-x}m  \log \left(\frac {r^2}{(c-1/2)^2} \right)\\
      &=\frac{x'+2(x-x')+2(n_m-x)}{m} \, \log \left(\frac r{c-1/2} \right)\\
      &=\frac{2n_m-x'}{m}  \, \log \left(\frac r{c-1/2} \right). 
  \end{align*}
  If all of the zeros were unpaired, then $2n_m-x'=n_m$, and this
  argument would reduce to McCurley's. Fortunately, by construction
  $x'\leq mE_\delta/\pi$, and so
  \[S_m(c,r) \geq \frac{2n_m - mE_\delta/\pi}{m} \, \log \left( \frac r{c-1/2} \right),\]
  whence
  \[\frac{n_m}{m} \leq \frac{S_m(c,r)}{2\log \left( r/(c-1/2) \right)} +
    \frac{E_\delta}{2\pi}.\] Lemma~\ref{lem:n/m meaning} completes this
  proof.
\end{proof}

No effort was made to use the pairs of zeros in
$(\frac12+\delta,\sigma_1]$. This is because the pairs of such zeros
may lie outside $(c-r,c+r)$ and so may not contribute to $S$. With a
stronger assumption about $r$, we can guarantee that the pair should
get counted and obtain a slightly stronger but more involved
bound. 
In practice, the paired zero is very close
to the edge of $D(c,r)$, and so the improvement is very slight except
for tiny $q$ and $T$, which we may handle by direct computation
anyway.

\begin{prop}[Backlund's trick, inelegant version]\label{prop:Backlund2}
  Let $c$ and $r$ be real numbers, and set $\sigma_1 \coloneqq  \frac 12 + \sqrt{2}(c-\frac 12)$ and
  $\delta\coloneqq 2c-\sigma_1-\frac12$. Further, let
  $$
  S_m(c,r)\coloneqq \frac 1m \sum_{z\in \Zero_{m}(D(c,r))} \log \frac r{|z-c|},
  $$
 $E_\delta\coloneqq E(a_\tau,\delta,T)$, and
  $E_{\sigma_1}\coloneqq  E(a_\tau,\sigma_1-\tfrac12,T)$. 
 If $r>(1+\sqrt2)(c-\frac12)$,
 $c>1$ and $\frac14 \leq \delta < \sigma_1 < \frac92$, then
 \begin{equation*}
    \left| \arg L(\sigma+iT,\tau) \Big|_{\sigma=\sigma_1}^{1/2}\right|
    \leq \frac{\pi \, S_m(c,r)}{2\log r/(c-1/2) } + \frac {E_\delta}2 +\frac{\pi}{m} +\frac{E_{\sigma_1}-E_\delta}{2} \left(1-\frac{\log(1+\sqrt
        2)}{\log r/(c-1/2)}\right).
  \end{equation*}
\end{prop}

\begin{proof}
    The proof is essentially identical to that of the preceding proposition. We arrive at the inequality
  \begin{align*}
  S_m(c,r)&\ge \frac {x_1'}m \log \left( \frac{r}{c-1/2} \right) + \frac {x_1-x_1'}m \log \left( \frac{r^2}{(c-1/2)^2} \right) + \frac {x_2'}m \log \left( \frac{r}{\sigma_1-c} \right) \\
&= \frac{2n_m}{m}\log \left( \frac{r}{c-1/2}  \right) - \frac{x_1'}{m}\log \left( \frac{r}{c-1/2} \right),
\end{align*}
from which this proposition follows, again upon invoking Lemma~\ref{lem:n/m meaning}.
\end{proof}

\subsection{The Jensen integral} \label{Jennie}

To use Proposition~\ref{prop:Jensen}, we require an explicit function $F_{c,r}(\theta)$ that will bound the quantity $\frac1m \log |f_m(c+re^{i\theta})|$ and that is even as a function of~$\theta$.

We begin by quoting some useful bounds on $L(s,\tau)$. The first bound~\eqref{eq:Lright} is straightforward from the Euler products
for $L(s,\tau)$ and $\zeta(s)$. The second bound~\eqref{eq:Lmiddle} is
Rademacher's convexity bound~\cite{Rademacher}. The
third bound~\eqref{eq:Lleft} follows from the second with $\eta=-\sigma$,
although it is actually a primary ingredient in Rademacher's proof
of equation~\eqref{eq:Lmiddle}.
\begin{lemma}
  Let $\tau$ be a primitive character with modulus $q>1$.  Fix a
  parameter $\eta\in (0,\frac12]$ and let $s=\sigma+it$.  If $\sigma \ge 1+\eta$, then
  \begin{equation}\label{eq:Lright}
    |L(s,\tau)|
    \leq \zeta(\sigma).
  \end{equation}
  If $-\eta \leq \sigma \leq 1+\eta$, then
  \begin{equation}\label{eq:Lmiddle}
    |L(s,\tau)|
    \leq \zeta(1+\eta) \left( \frac{q}{2\pi}\cdot |s+1| \right)^{(1+\eta-\sigma)/2}.
  \end{equation}
  If $-\frac12 \leq \sigma \leq -\eta$, then
  \begin{equation}\label{eq:Lleft}
    |L(s,\tau)|
    \leq \zeta(1-\sigma) \left( \frac{q}{2\pi}\cdot |s+1| \right)^{\frac12-\sigma}.
  \end{equation}
\end{lemma}

We can leverage Rademacher's argument to also provide bounds to the left of $\sigma=-1/2$. For a real number $x$, let $[x]$ be the integer closest to $x$, choosing the one closer to 0 if there are two integers equally close to $x$. We note that $[-x]=-[x]$.

For $-\frac12 \leq \sigma<0$, the following result reduces to equation~\eqref{eq:Lleft}.
\begin{theorem}\label{thm:Lboundleft}
  Let $\tau$ be a primitive character with modulus $q>1$.  Suppose
  $s=\sigma+it$, with $\sigma< 0$. Then
  \begin{equation}\label{Rademacher+}
    |L(s,\tau)| \leq \zeta(1-\sigma)
    \left(\frac{q}{2\pi}\right)^{\frac12-\sigma}\cdot\big| s-[\sigma]
    + 1 \big|^{\frac12-\sigma+[\sigma]} \cdot \prod_{j=1}^{-[\sigma]}
    \left|s+j-1\right|.
  \end{equation}
\end{theorem}

\begin{proof}
Let $a=a_\tau$ be the sign and $q$ the modulus (and conductor) of
  $\tau$.  From the functional equation~\eqref{Functional Equation}, it follows that 
  \[|L(s,\tau)|
    =\left(\frac{q}{\pi}\right)^{-\sigma+1/2}\cdot|L(1-s,\bar\tau)|\cdot
    \left| \frac{\Gamma(\frac a2+\frac{1-s}{2})}{\Gamma(\frac a2 +
        \frac s2)}\right|.\] As $\sigma<0$, we may apply 
  Lemma~\ref{lem:EPB} to conclude that
  \[|L(1-s,\bar\tau)| \leq \zeta(1-\sigma).\] We are therefore left with a
  ratio of gamma functions to bound.
To do this, we appeal to Euler's reflection formula
\[
\Gamma(1-z) \Gamma(z)=\frac{\pi}{\sin (\pi z)}
\]
 and Legendre's duplication formula
\[
\Gamma(z) \Gamma(z+\tfrac{1}{2})=2^{1-2 z} \sqrt{\pi} \cdot \Gamma(2 z) .
\]
It follows, as $a\in\{0,1\}$, that
\begin{align*}
\frac{\Gamma({\frac a2+\frac{1-s}2})}{\Gamma({\frac a2+\frac{s}2})} &= \frac{\Gamma(1-\frac{s+1-a}2)}{\Gamma({\frac a2+\frac{s}2})} \\
&= \frac\pi{\Gamma({\frac a2+\frac{s}2}) \Gamma(\frac{s+1-a}2) \sin(\frac\pi2(s+1-a))} \\
&= \frac{\sqrt\pi}{2^{1-s} \Gamma(s) \sin(\frac\pi2(s+1-a))}.
\end{align*}
Replacing $s$ by $s+k$ and $a$ by $b\in\{0,1\}$, we obtain
\begin{align*}
\frac{\Gamma(\frac b2+\frac{1-(s+k)}2)}{\Gamma(\frac b2+\frac{s+k}2)} &= \frac{\sqrt\pi}{2^{1-(s+k)} \Gamma(s+k) \sin(\frac\pi2(s+k+1-b))}.
\end{align*}
Comparing the last two equations reveals that
\begin{align*}
\frac{\Gamma({\frac a2+\frac{1-s}2})}{\Gamma({\frac a2+\frac{s}2})} &= \frac{\Gamma(\frac b2+\frac{1-(s+k)}2)}{\Gamma(\frac b2+\frac{s+k}2)} \frac{2^{1-(s+k)}}{2^{1-s}} \frac{\Gamma(s+k)}{\Gamma(s)} \frac{\sin(\frac\pi2(s+k+1-b))}{\sin(\frac\pi2(s+1-a))} \\
&= \frac{\Gamma(\frac b2+\frac{1-(s+k)}2)}{\Gamma(\frac b2+\frac{s+k}2)} 2^{-k} \bigg( \prod_{j=1}^{k} (s+j-1) \bigg)  \frac{\sin(\frac\pi2(s+k+1-b))}{\sin(\frac\pi2(s+1-a))} .
\end{align*}
If we now choose $b\equiv a+k\mod2$, then the last factor becomes simply $\pm1$ and hence the desired inequality follows upon choosing $k=-[\sigma]$, bounding ${\Gamma(\frac b2+\frac{1-(s+k)}2)}/{\Gamma(\frac b2+\frac{s+k}2)}$ trivially 
and taking absolute values.  
\end{proof}

 It is worth observing that equation~\eqref{Rademacher+} is precisely  the inequality that Rademacher uses (together with
  equalities from the functional equation), so any improvement to this
  bound would also yield an improvement to Rademacher's bound in the
  critical strip, and vice versa.
  
\begin{lemma}\label{lem: f bound RHP}
  Fix a parameter $\eta\in(0,\frac12]$ and $T>0$. Write $s=\sigma +it$. If
  $\sigma \ge 1+\eta$, then
  \begin{equation}
    \frac 1m \log | f_m(s)| \leq \log \zeta(\sigma). \label{eq:fright}
  \end{equation}
  If $-\eta \leq \sigma \le 1+\eta$, then
  \begin{equation}
    \frac1m \log|f_m(s)| \leq \log\zeta(1+\eta)+\frac{1+\eta-\sigma}{2}\log\frac {q}{2\pi} +\frac{1+\eta-\sigma}{4}\log\big(
    (\sigma+1)^2+(|t|+T)^2\big). \label{eq:fmiddle}
  \end{equation}
  If $\sigma\leq -\eta$, then
  \begin{multline}
    \frac1m \log|f_m(s)| \leq \log\zeta(1-\sigma)+\frac{1-2\sigma}{2} \log\frac {q}{2\pi} +\frac{1-2\sigma+2[\sigma]}{4} \log\big( (1+\sigma-[\sigma])^2+(|t|+T)^2\big) \\
    +\frac12 \sum_{k=1}^{-[\sigma]}\log\big( (\sigma+k-1)^2+(|t|+T)^2
    \big). \label{eq:fleft}
  \end{multline}
\end{lemma}
  
  \begin{proof}
    For $\sigma\ge 1+\eta>1$, the real parts of $s+T i$ and $s-T i$
    are both at least $1+\eta$, and so Lemma~\ref{lem:EPB} gives
    \begin{align*}
      |f_m(s)| &\coloneqq  \left| \frac1{2}\left( L(s+iT,\tau)^m + L(s-iT,\bar\tau)^m\right) \right| \\
               & \leq \frac1{2} \left( \left| L(s+iT,\tau)\right|^m + \left| L(s-iT,\bar\tau)\right|^m \right) \\
               &\leq \frac 12 \left( \zeta(\sigma)^m + \zeta(\sigma)^m \right) =\zeta(\sigma)^m.
    \end{align*}
    Taking real logarithms yields~\eqref{eq:fright}.
  
    For the second claim, with $-\eta \leq \sigma \leq 1+\eta$, we
    have
    \begin{align*}
      |f_m(s)| &\coloneqq  \left| \frac1{2}\left( L(s+iT,\tau)^m + L(s-iT,\bar\tau)^m\right) \right| \\
               &\leq \frac12 \left( \left| L(s+iT,\tau)\right|^m + \left| L(s-iT,\bar\tau)\right|^m \right) \\
               & \leq \frac12 \left( \bigg(\zeta(1+\eta) \left(  \frac{q}{2\pi}|s+T i+1| \right)^{(1+\eta-\sigma)/2}\bigg)^m \right. \\
               &\qquad \qquad+ \left. \bigg(\zeta(1+\eta) \left(  \frac{q}{2\pi} |s-T i+1| \right)^{(1+\eta-\sigma)/2}\bigg)^m \right).
    \end{align*}
    Writing $s=\sigma+ti$, with $T>0$, we have\footnote{As we are
      ultimately concerned with $m\to\infty$, we could care only about
      whichever one of $|s\pm Ti+1|$ is larger, saving a factor of
      2. But also, as $m\to\infty$ and we  ultimately care about $\frac 1m \log|f_m(s)|$, a factor of 2 is
      irrelevant.}
    \[ |s \pm T i+1|^2 = (\sigma+1)^2+(t\pm T)^2 \leq
      (\sigma+1)^2+(|t|+T)^2. \] Thus,
    \[|f_m(s)| \leq \bigg(\zeta(1+\eta) \left( \frac{q}{2\pi} \sqrt{
          (\sigma+1)^2+(|t|+T)^2}
      \right)^{(1+\eta-\sigma)/2}\bigg)^m,\] and routine 
    manipulation of logarithms completes the lemma.

    Finally, for $-\frac12 \leq \sigma<-\eta$, we have $-[\sigma]=0$, and so line~\eqref{eq:fleft} follows
    from~\eqref{eq:fmiddle} upon setting $\eta=-\sigma$ and noting that the summation over $k$ is empty. Assume now that
    $\sigma< -\frac 12$, so that $[\sigma]\leq -1$. As the exponent
    $[\sigma]-\sigma+1/2$ is nonnegative, we see that the bound from
    Theorem~\ref{thm:Lboundleft} is monotone increasing for
    $\Im(s)>0$ (with $\Re(s)$ fixed). We apply Theorem~\ref{thm:Lboundleft} to $s$ with
    imaginary parts $t+T$ and $t-T$, and $|t\pm T| \leq |t| +
    T$. Thus,
    \begin{multline*}
      |f_m(s)| = \left| \frac12 \left( L(s+iT,\tau)^m + L(s-iT,\bar\tau)^m\right) \right| \\
               \leq \zeta(1-\sigma)^m \left(\frac{q}{2\pi}\right)^{m(\frac12-\sigma)} \left((\sigma-[\sigma]+1)^2+(|t|+T)^2 \right)^{m\cdot\frac12\cdot (\frac12-\sigma)} \\
               \cdot \left(\prod_{k=1}^{-[\sigma]} \big( (\sigma+k-1)^2+(|t|+T)^2 \big) \right)^{1/2}.
    \end{multline*}
    Taking real logarithms completes the proof.
  \end{proof}

\begin{definition}\label{def:L} We set 
  \[L_{j}(\theta) \coloneqq  \log\frac{(j+c+r\cos\theta)^2+(|r\sin\theta|+T)^2}{(T+2)^2}.\]
  We note that $L_j(\theta)$ depends on $c$, $r$ and $T$. If we suppose that  $0\leq \theta \leq \pi$ and $T\geq 5/7$, then, from the inequality $\log x \leq x-1$, we find that $L_j(\theta) \leq L_j^\star(\theta) / (T+2)$, where 
  \[L_j^\star(\theta) \coloneqq  2r\sin\theta-4 + \tfrac 7{19} \left( (j+c+r \cos\theta )^2 + (r \sin\theta-2)^2\right).\]
\end{definition}

\begin{definition}\label{def:F}
  Let $\sigma=c+r\cos\theta$ and $t=r\sin\theta$, where $-\pi \leq \theta \leq \pi$. Define
\[
F_{c,r}(\theta) \coloneqq \begin{cases}
\log \zeta(\sigma), &\text{if } \sigma \ge 1+\eta, \\
\displaystyle \log\zeta(1+\eta) + \frac{1+\eta-\sigma}{2}\ell + \frac{1+\eta-\sigma}{4}L_1(\theta), &\text{if } -\eta \leq \sigma \leq 1+\eta, \\
\displaystyle \log\zeta(1-\sigma) +\frac{1-2\sigma}{2} \ell
    + \frac{1-2\sigma+2[\sigma]}{4} L_{1-[\sigma]}(\theta)
    +\frac12\sum_{k=1}^{-[\sigma]} L_{k-1}(\theta), &\text{if } \sigma < -\eta.
\end{cases}
\]
  We note that $F_{c,r}(\theta)$ depends on $q,T$ and $\eta$
  implicitly. Usefully, $F_{c,r}$ is a continuous,  even function of $\theta$.
  Figure~\ref{fig:jensenintegrand}~shows the function $F_{1.2,1.9}(\theta)$ for
  $T=1$ and $q=10^6$ and $\eta = 0.141$.
  \begin{figure}
    \includegraphics[width=10cm]{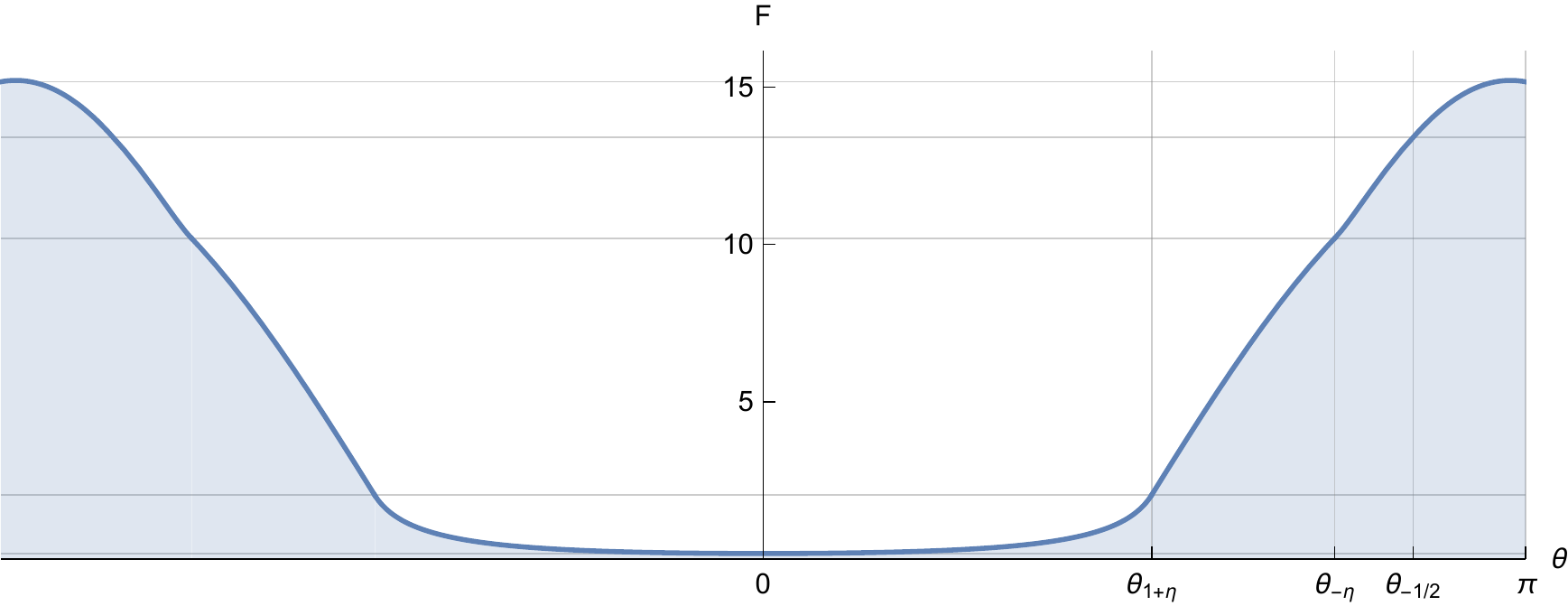}
    \caption{The function $F_{1.2,1.9}(\theta)$ for $T=1$,
      $q=10^6$, and $\eta=0.141$. With these parameters,
      \(\int_0^\pi F_{1.2,1.9}(\theta)\,d\theta \doteq
      16.37 \).}\label{fig:jensenintegrand}
  \end{figure}
\end{definition}

\begin{definition}\label{def:theta}
  If $c$ and $r$ are real numbers, with $r$ positive,  define $\theta_\sigma$ as
  \[\theta_\sigma \coloneqq  \begin{cases} 0, & \mbox{ if } c+r \leq \sigma; \\
      \arccos\frac{\sigma-c}{r}, & \mbox{ if } c-r \leq \sigma \leq c+r;\\ \pi, &
     \mbox{ if } \sigma\leq c-r.\end{cases}\] 
We remark that for $c-r\leq \sigma \leq c+r$, we have $c+r\cos\theta_\sigma = \sigma$.   
\end{definition}

\begin{definition}\label{def:kappas}
We also define
  \[ \kappa_1 \coloneqq  \left(\theta_{-\eta} - \theta_{1+\eta}\right)
    \frac{1+\eta-c}{2}-\left(\pi-\theta_{-\eta}\right)
    \left(c-\frac{1}{2}\right)+\frac{r \left(\sin
        \theta_{-\eta}+\sin\theta_{1+\eta}\right)}{2 }.\]
  For a positive integer $J_1$ (we will actually take $J_1=64$), set
    \[\kappa_2(J_1) \coloneqq \frac{\pi}{4J_1} \bigg(\log\zeta(c+r)+2\sum_{j=1}^{J_1-1}
      \log\zeta(c+r\cos\tfrac{\pi j}{2J_1}) \bigg).\]
  For a positive integer $J_2$ (we will actually take $J_2=24$), set
    \[\kappa_3(J_2) \coloneqq 
      \frac{\pi-\theta_{1-c}}{2J_2}\bigg(\log\zeta(1-c+r)+2\sum_{j=1}^{J_2-1}
      \log\zeta\big(1-c-r\cos(\tfrac{\pi j}{J_2}+(1-\tfrac j{J_2}) \theta_{1-c})\big)\bigg).\]    
Finally, define
	\begin{align*}
	\kappa_4 
	&\coloneqq  \frac 1{4} \int_{\theta_{1+\eta}}^{\theta_{-\eta}}(1+\eta-\sigma) L_1^\star(\theta)\,d\theta, \\ 
	\kappa_5
	&\coloneqq   \frac 1{4} \int_{\theta_{-\eta}}^{\theta_{-1/2}}(1-2\sigma)L_1^\star(\theta)\,d\theta, \; \mbox{ and}\\
	\kappa_{6,j}
	&\coloneqq  \frac 14 \int_{\theta_{-j+1/2}}^{\theta_{-j-1/2}} \left((1-2\sigma-2j)L_{j+1}^\star + 2\sum_{k=1}^j L_{k-1}^\star(\theta)\right)\,d\theta.
	\end{align*}
	Note that the integrands involved here are polynomials in $\sin\theta$ and $\cos\theta$, and so we can evaluate the integrals exactly. These evaluations are not enlightening to examine, but they are computationally important. For details, the reader may consult the files {\it ZerosOfLFunctions-Largeell.nb} and {\it ZerosOfLFunctions-Middleell.nb} at
\begin{center}
\texttt{\href{http://www.nt.math.ubc.ca/BeMaObRe2/}{\url{http://www.nt.math.ubc.ca/BeMaObRe2/}}}
\end{center}
\end{definition}

%
%

From the fact that $F_{c,r}$ is an even function, the Jensen integral is then bounded
as
\[\frac{1}{2\pi}\int_{-\pi}^{\pi} \frac 1m \log
  |f_m(c+re^{i\theta})|\,d\theta \leq \frac 1\pi \int_0^\pi
  F_{c,r}(\theta)\,d\theta. \] 
  We evaluate the main term of this integral (as $q$
or $T$ go to $\infty$)  with the fundamental theorem of
calculus, while the minor terms require labourious bounding.
\begin{prop}\label{prop:Integrals}
  Let $c,r$ and $\eta$ be positive real numbers satisfying
  \begin{equation} \label{zany}
 1+\eta \leq c < r-\eta,
  \end{equation}
 and suppose that $q\ge 3$ and $T\geq 5/7$. Then $\int_0^\pi F_{c,r}(\theta)\,d\theta$ is at most
  \begin{multline*}
    \kappa_1 \ell +(\theta_{-\eta} - \theta_{1+\eta})\log\zeta(1+\eta) 
                    + \int_0^{\theta_{1+\eta}}\log\zeta(c+r \cos \theta)\,d\theta \\
                    +\int_{\theta_{-\eta}}^{\pi}\log\zeta(1-c-r \cos \theta)\,d\theta 
		+\frac{\kappa_4+\kappa_5}{T+2} +\frac{1}{T+2} \sum_{j=1}^\infty \kappa_{6,j}.
  \end{multline*}
\end{prop}

The infinite sum is cosmetic: if $j\geq r-c+\frac12$, then $\theta_{-j-1/2}=\theta_{-j+1/2}=\pi$ and so $\kappa_{6,j}=0$.

\begin{proof}
  By the hypothesized inequalities, we have
  $\theta_{-\eta}<\theta_{-1/2}$.  Rearranging terms, 
  \begin{align}
    \int_0^\pi F_{c,r}(\theta)\,d\theta
    = &\int_{0}^{\theta_{1+\eta}} \log\zeta(\sigma)\,d\theta + \int_{\theta_{-\eta}}^{\pi} \log\zeta(1-\sigma)\,d\theta \notag \\
      &+\int_{\theta_{1+\eta}}^{\theta_{-\eta}} \log\zeta(1+\eta)\,d\theta \label{eq:constant integral}\\
      &+ \ell \left( \int_{\theta_{1+\eta}}^{\theta_{-\eta}} \frac{1+\eta-\sigma}{2}\,d\theta + \int_{\theta_{-\eta}}^{\pi}\frac{1-2\sigma}{2} \,d\theta\right) \label{eq:ell terms}\\
      &+\int_{\theta_{1+\eta}}^{\theta_{-\eta}} \frac{1+\eta-\sigma}{4}\,L_1(\theta)\,d\theta
        + \int_{\theta_{-\eta}}^{\theta_{-1/2}} \frac{1-2\sigma}{4}\, L_1(\theta)\,d\theta \label{eq:first L} \\
      &+\sum_{j=1}^\infty \int_{\theta_{-j+1/2}}^{\theta_{-j-1/2}} \bigg( \frac{1-2\sigma-2j}{4}\,L_{j+1}(\theta)+\frac12\sum_{k=1}^jL_{k-1}(\theta) \bigg) \,d\theta. \label{eq:other L}
  \end{align}
  The integrand on line~\eqref{eq:constant integral} is constant; the
  integral is $(\theta_{-\eta} - \theta_{1+\eta})\log\zeta(1+\eta)
  $. The integrals on line~\eqref{eq:ell terms} are exactly
  $\kappa_1$.
  
  For $\theta_{1+\eta} < \theta < \theta_{-\eta}$, we have $-\eta<\sigma < 1+\eta$ and so $1+\eta-\sigma>0$. For $\theta_{-\eta} < \theta < \theta_{-1/2}$, we have $-\frac12 < \sigma <\eta <0$, whence $1-2\sigma>0$. Thus, line~\eqref{eq:first L} is bounded (using $T\geq 5/7$) by
  \[\int_{\theta_{1+\eta}}^{\theta_{-\eta}} \frac{1+\eta-\sigma}{4}\,L_1^\star(\theta)\,d\theta
          + \int_{\theta_{-\eta}}^{\theta_{-1/2}} \frac{1-2\sigma}{4}\, L_1^\star(\theta)\,d\theta = \frac{\kappa_4+\kappa_5}{T+2}.\]
  Likewise, if $\theta_{-j+1/2} < \theta < \theta_{-j-1/2}$, then $-j-\frac12 < \sigma < -j+\frac12$, and so $$
  1-2\sigma-2j>0
  $$ 
  and we can appeal to the inequality $L_j(\theta) \leq L_j^\star(\theta)/(T+2)$. This bounds line~\eqref{eq:other L} by
  \[\frac{1}{T+2} \sum_{j=1}^\infty \kappa_{6,j},\] as claimed.
\end{proof}

Our goal in the remainder of this section is to provide upper bounds for the two integrals appearing in  the statement of Proposition \ref{prop:Integrals}. In both cases, these bounds will take the form of a small finite sum of reasonably manageable (that is, easily optimized) functions.

\begin{lemma} \label{lem1}
  Let  $c,r$ and $\eta$ be positive real numbers satisfying \eqref{zany}, $\sigma=c+r\cos\theta$ and $J_1$ be a positive integer.  If 
  $\theta_{1+\eta}\leq 2.1$, then
  \begin{equation*}
    \int_{0}^{\theta_{1+\eta}} \log\zeta(\sigma)\,d\theta
    \leq \frac{\log\zeta(1+\eta)+\log\zeta(c)}{2}(\theta_{1+\eta} -\frac \pi2 )+\frac{\pi}{4J_1} \log\zeta(c)+\kappa_2(J_1).
  \end{equation*}
\end{lemma}

\begin{proof}
  As the map $\theta \mapsto \log\zeta(c+r \cos \theta)$ is increasing
  for $0\leq \theta \leq \theta_{1+\eta}$, we could use right
  endpoints to overestimate the integral
  $\int_{0}^{\theta_{1+\eta}} \log\zeta(\sigma)\,d\theta$. We can get
  the needed accuracy using many fewer terms, however, by showing that
  the map is convex, whereby the trapezoid rule provides an
  overestimate.

  To see that the map is convex, observe that
  \begin{align}
    \frac{d^2}{d\theta^2}
    \log\zeta(c+r\cos\theta)
    &= \frac{d^2}{d\theta^2}\sum_p - \log(1-p^{-c-r\cos\theta}) \notag \\
    &= \sum_p \frac{r\log p}{(1-p^{c+r\cos\theta})^2} \left(p^{c+r\cos\theta}(\cos\theta+r \log (p) \sin^2\theta)-\cos\theta\right) \notag\\
    &\geq \sum_p \frac{r\log p}{(1-p^{c+r\cos\theta})^2} \left(p^{c+r\cos\theta}(\cos\theta+ \log (2) \sin^2\theta)-\cos\theta\right).\label{eq:logzeta pieces}
  \end{align}
  Here, the sums are over primes $p$. Since $0\leq \theta \leq  \theta_{1+\eta} \leq 2.1$, Definition \ref{def:theta} and (\ref{zany}) together imply that 
  $$
  1 < 1+\eta = c+r \cos \theta_{1+\eta} \leq c+r \cos \theta
  $$
  (which in particular justifies the use of the Dirichlet series for $\log \zeta$),
 while $0\leq \theta \leq 2.1$ yields the inequality
  $\cos\theta+\log (2) \sin^2\theta >0$. It thus follows that
  \[ p^{c+r\cos\theta}(\cos\theta+\log (2) \sin^2\theta ) >
    (\cos\theta+\log (2) \sin^2\theta ),\] whence
  \[p^{c+r\cos\theta}(\cos\theta+ \log (2) \sin^2\theta)-\cos\theta>
    \log (2) \sin^2\theta,\] and so \eqref{eq:logzeta pieces} is
  positive, term-by-term.

 Singling out the part between $\theta_{c}=\pi/2$ and
  $\theta_{1+\eta}\geq \theta_c$, we obtain the claimed bound.
\end{proof}

In a nearly identical fashion, we prove the next lemma.  The hypotheses on $c,r$ and $\eta$ guarantee that $c-r<1-c\leq -\eta$, whence $\pi > \theta_{1-c} \geq \theta_{-\eta}$. 
\begin{lemma} \label{lem2}
 Let $J_2$ be a positive integer.  If $r> 2c-1$ and $1+\eta\leq c$, then
$$
    \int_{\theta_{-\eta}}^{\pi} \log\zeta(1-\sigma)\,d\theta \leq
    \frac{\log\zeta(1+\eta)+\log\zeta(c)}{2}(\theta_{1-c}-\theta_{-\eta})
    +\frac{\pi-\theta_{1-c}}{2J_2}\log\zeta(c)+\kappa_3(J_2).
$$
\end{lemma}

\section{Assembling the bound} \label{sec-assembling}

We begin this section by describing  how to assemble the results in the proceeding sections to produce  an explicit bound for $\ell\geq 27.02$. In Section~\ref{subsec-middle ell}, we adjust this argument to treat values for $\ell$ with $5.98 \leq \ell  \leq 28$. Finally,  in Section~\ref{subsec-small ell} we outline the rigorous  explicit computations of zeros that  allows us  to handle small $\ell$ with $\ell\leq 6$.

\subsection{Large Values of $\ell$}\label{subsec-large ell}

Let us assume that $\ell\geq 27.02$ and set
	\begin{align*}
	\ell&\coloneqq \log \frac{q(T+2)}{2\pi} \subseteq[27.02,\infty),&\quad \eta&\coloneqq \frac{18}{10+9\ell} \subseteq(0,0.08),\\
	c &\coloneqq  1+\frac{391}{74 \ell+683} \subseteq (1,1.15),&\quad r & \coloneqq  \frac{149}{140}+\frac{769}{30 \ell+512} \subseteq (1.06,1.65),\\
	\sigma_1 &\coloneqq c+ \frac{(c-1/2)^2}{r} \subseteq (1.23,1.40), & \quad \delta &\coloneqq  2c-\sigma_1-\frac12 \subseteq (0.26,0.40).
	\end{align*}
These definitions guarantee the chain of inequalities
   \[-\frac12 < c-r<1-c<-\eta < 0 < 1 < 1+\eta < c < \sigma_1 < c+r,\]
which is simply (\ref{zany}) along with the extra condition that $c-r<1-c<-\eta$. 
We set $E_\delta\coloneqq E(a_\chi,\delta,T)$. The values for $c$ and $r$ were chosen after extensive numerical work, with $\frac{149}{140}=1.06429\cdots$ being a good approximation to our numerically determined ``ideal'' value  of $r$. Numerical work suggests that we should choose $c=1+O(\frac{\log\ell}{\ell})$, but the  improvement in the final values is slight, while the added complexity in  producing a bound would  be considerable.

The value of $\eta$ can be motivated, however, and some words on why we define $\eta$ in this way are appropriate. To apply
Lemma~\ref{lem: f bound RHP}, we require $\eta\leq 1/2$. In an
ideal world, we could choose $\eta$ optimally for each $\sigma$, so
as to make the right side of~\eqref{eq:fmiddle} as small as
possible. Experiments indicate that the numerical advantage in doing
so is slight, albeit noticeable, and not justifying the added
complexity. The derivative with respect to $\eta$ of~\eqref{eq:fmiddle} at
$\sigma=1/2$ is
	\[\frac{1}{2} \log\frac{q}{2 \pi }+\frac{1}{4} \log \left((\sigma
	+1)^2+(t+T)^2\right)+\frac{\zeta '(1+\eta)}{\zeta (1+\eta)}.\]
For $\eta$ between~$0$ and $1/2$, we know that
$\frac{\zeta'}{\zeta}(1+\eta)+\frac1\eta$ is nearly linear,
decreasing from $\gamma \doteq 0.577216$ to just below $1/2$; we choose $5/9$ as a convenient rational in the desired range. We handle small $\ell$ (for which $\eta$ is near
$1/2$) by direct computation of zeros, so we find it reasonable to
replace $\frac{\zeta'}{\zeta}(1+\eta)$ with
$\frac59-\frac1\eta$. The value of $t$ will cover a range, but $t=2$
seems roughly typical. The critical value of $\eta$ is then
estimated as the solution to
	\[\frac{1}{2} \log\frac{q}{2 \pi }+\frac{1}{4} \log
	\left((2+T)^2\right)+\frac59 - \frac1\eta=0,\]
which is $\eta = \frac{18}{10+9\ell}$. Setting $\eta$ in this way allows us to give the single bound in Theorem~\ref{thm:Main} instead of a table of bounds for various settings of $\eta$ (as in
\cite{Trudgian}) or a bound that depends continuously on $\eta$ (as in \cite{McCurley}).

With these choices of parameters, we now return to inequality (\ref{eq:One arg term left}). Appealing to Propositions \ref{prop:Jensen} and \ref{prop:Backlund}, and letting
$m\to\infty$, we find that
	\begin{multline} \label{with integral term}
	\bigg| N(T,\chi)  - \left( \frac{T}{\pi} \log\frac{qT}{2\pi e} -\frac{\chi(-1)}{4}\right) \bigg| 
	\leq
	|g(a,T)| + \frac 2\pi \log\zeta(\sigma_1) 
	+ \frac{E_\delta}\pi \\
	+ \frac{\log\zeta(c) - \log\zeta(2c)}{\log r/(c-1/2)}
	+ \frac{1/\pi}{\log r/(c-1/2) } \int_0^\pi F_{c,r}(\theta)\,d\theta.
	\end{multline}
Lemma~\ref{gamma prop} bounds $g(a,T)$ and Lemma~\ref{lem:E bound} bounds $E_\delta$. Combining those bounds, whose sum is monotone in $d$ and rational in $T$, we can prove that
  \[|g(a,T)| + \frac{E_\delta}\pi  \leq \frac{1}{14 (T-{1}/{5})}+\frac{1}{2^{10}}.\]
Using interval analysis,
  \[\frac 2\pi \log\zeta(\sigma_1) -\frac{\log\zeta(2c)}{\log r/(c-1/2)} 
  \leq \frac{178 \ell^2+17909 \ell+80807}{4 \left(128 \ell^2+9637 \ell+164296\right)}.
  \]

We now consider the integral term in equation~\eqref{with integral term}. We apply Proposition~\ref{prop:Integrals} to break the integral $\int F_{c,r}$ into pieces. With our settings for $c,r,\eta$ and bound on $\ell$, the hypotheses are satisfied and  $\theta_{j-1/2}=\pi$ for all $j\ge 0$, whence the ``infinite'' sum is 0. We use Lemmata~\ref{lem1} and \ref{lem2} to bound the pieces. The main term is bounded as
	\[
	\frac{\kappa_1 \ell /\pi}{\log r/(c-1/2)} 
	\leq
	\frac{238413}{2^{20}}\ell + \frac{798 \ell^2+135589 \ell+80396}{16 \left(32 \ell^2+3105 \ell+38735\right)}.
	\]
Collecting the various $\log\zeta(c)$ terms, we have a total of
	\[\left(\frac{\theta _{1-c}-\theta _{-\eta }+\theta _{1+\eta}}{\pi }-\frac{\theta _{1-c}}{\pi  J_2}+\frac{1}{2 J_1}+\frac{1}{J_2}+\frac{3}{2}  \right)\frac{\log \zeta (c)}{2 \log \left(\frac{r}{c-1/2} \right)},\]
where $\theta_\sigma$ is defined in Definition~\ref{def:theta}.
With $J_1=64, J_2=24$, these terms contribute at most
	\[\frac{-1135 \ell^2-214796 \ell+149201}{512 \ell^2+75117 \ell+496726}+\frac{\ell}{2^{20}}+\frac{1365 \log (\ell+1)}{2^{10}}\]
Collecting the various $\log\zeta(1+\eta)$ terms, we have a total of
	\begin{equation*}
	\left(\frac{\theta _{1-c}+\theta _{-\eta }-\theta _{1+\eta}}{\pi } - \frac 12\right)
	\frac{\log \zeta (1+\eta)}{2 \log \frac{r}{c-1/2}} 
	\leq
	\frac{-182 \ell^2-118430 \ell+79045}{512 \ell^2+91562 \ell+599789}+\frac{\ell}{2^{22}}+\frac{529 \log (\ell+1)}{2^{10}}.
	\end{equation*}
For the absolutely bounded terms, we obtain the inequalities
	\begin{align*}
	\frac{\kappa_2/\pi}{\log r/(c-1/2)} 
	&\leq \frac{635}{1024}-\frac{9 (113745 \ell+25384532)}{64 \left(512 \ell^2+150141 \ell+7149852\right)},\\
	\frac{\kappa_3/\pi}{\log r/(c-1/2)} 
	&\leq \frac{491}{1024}-\frac{3346893 \ell+33179656}{512 \left(512 \ell^2+21113 \ell+208616\right)}.
	\end{align*}
Since the $O(1/T)$ terms contribute
    \begin{align*}
    \frac{1/\pi}{\log \left( \frac{r}{c-1/2} \right)} \frac{\kappa_4}{T+2} 
    &\leq \frac{1}{T+2}\left(\frac{-50 \ell^2-1411 \ell+18281}{512 \ell^2+63962 \ell+800695}\right)\\
    \frac{1/\pi}{\log \left(\frac{r}{c-1/2} \right)}  \frac{\kappa_5}{T+2}
    &\leq  \frac{1}{T+2}\left( \frac{-42 \ell^2-15293 \ell-961048}{512 \ell^2+113665 \ell+3255348} \right),
    \end{align*}
we are led to conclude that
    \[|g(a,T)|+ \frac{E_\delta}{\pi}+ \frac{1/\pi}{\log \frac{r}{c-1/2} } \frac{\kappa_4}{T+2} +\frac{1/\pi}{\log \frac{r}{c-1/2} }  \frac{\kappa_5}{T+2} \leq \frac{75}{2^{10}}.\]
The remaining terms involve only $\ell$, and we find (rigorously, as with all the inequalities in this article) that they are at most
   \[0.22737 \ell  +  2  \log(1+\ell) -  0.5.\]
This establishes Theorem~\ref{thm:Main} for $\ell  \geq 27.02$.

\subsection{Middle Values of $\ell$}\label{subsec-middle ell}
For $5.98 \leq \ell \leq 28$,  we set
\[c\coloneqq  1+\frac{505}{111\ell+430} \; \mbox{ and } \;  r\coloneqq  \frac{149}{140}+\frac{747}{36 \ell+283},\]
and find that 
$$
-\frac32 \leq c-r \leq -\frac12, \; c\geq 1+\eta, \; \theta_{1+\eta}\leq 1.62 \; \mbox{ and } \; r\geq 2c-1. 
$$
A similar fully rigorous analysis yields
\[\bigg| N(T,\chi)  - \left( \frac{T}{\pi} \log\frac{qT}{2\pi e} -\frac{\chi(-1)}{4} \right) \bigg| \leq 0.22737\ell +2\log(1+\ell)-0.5.\]

\subsection{Small Values of $T$ and $\ell$}\label{subsec-small ell}
We first attempted to use Rubinstein's {\tt LCALC}, and then gp/PARI, to compute all zeros of all $L$-functions up to conductor $10000$ and  $\ell \coloneqq  \log\frac{q(T+2)}{2\pi} \leq 8$. However, both programs were found to miss pairs of zeros occasionally. Using {\tt  Arb} for interval arithmetic, for each primitive character (we actually only concern ourselves with one from each conjugate pair) with conductor $1<q<935$ and $\ell  \leq 6$, we rigorously bounded the expression in equation~\eqref{eq:von Mangold}. To do so we used the identity
\begin{align*}
 \arg L \left(\frac12 + iT,\chi \right) &= \arg L(3+iT,\chi) + \Im \int_{3+iT}^{1/2+iT} \frac{L'(s,\chi)}{L(s,\chi)}  \,ds 
\end{align*}
to rigorously bound $\arg L(\frac12 + iT,\chi)$; the term $\arg L(3+iT,\chi)$ is the principal value of the argument (it is easy to show from the Euler product that $|\arg L(3+iT,\chi)| \le \sum_p \arcsin p^{-3} < 0.176$).
We thereby computed $N(T,\chi)$ for some~$T$ greater than or equal to ${2\pi e^6}/{q}-2$ (for some characters, the integrand is highly oscillatory and it can be advantageous to let  $T$ be slightly larger).
Then,  again using  {\tt Arb} for the rigorous computation, we found the $L$-function zeros by locating sign changes in the appropriate Hardy $Z$-function. From this approach, we rigorously located (and stored)  every zero of every nontrivial primitive Dirichlet $L$-function with conductor at most $934$ and $\ell \leq 6$, accurate to within $10^{-12}$. By only considering one from each pair of complex characters and only the positive imaginary parts for real characters, we examined 80818 characters and found a total of 403272 zeros.


With this dataset, we have proved the following lemma. The code to generate the dataset (in C), to process the dataset into Mathematica format, and Mathematica code to pull the following lemma out of the data, are all available on the website. Additional commentary on each item is provided below.
\begin{lemma}\label{comp results}
	Let $1<q<935$, suppose that $\chi$ is a primitive character with conductor $q$ , and set $\ell\coloneqq  \log\frac{q(T+2)}{2\pi} $. Then
	\begin{enumerate}
		\item \label{comp:GRH}
			 All of the zeros of $L(s,\chi)$ with real part between $0$ and $1$ and imaginary part between $-{2 (e^6 \pi -q)}/{q}$ and ${2 (e^6 \pi -q)}/{q}$ have real part equal to~$1/2$.
		\item \label{comp:tiny}
			If $T\geq 0$ and $\ell \le 1.567$, then $N(T,\chi)=0$.
		\item \label{comp:Strong bound}
			If $T\geq 0 $ and $1.567 \le \ell \le 6$, then
		   	\begin{equation*}
				\left| N(T,\chi) - \left( \frac{T}{\pi} \log\frac{qT}{2\pi e} - \frac{\chi(-1)}{4} \right) \right| 
					\le \frac{\ell}{\log(2+\ell)}.
			\end{equation*}
		\item \label{comp:table values}
			Let $a\coloneqq (1-\chi(-1))/2$ and $T\in\{\frac 57, 1, 2\}$. Let $0\leq k \leq 4$, or $(T,k)=(2,5)$, or $(a,T,k)=(1,2,6)$. Let $q_a(T,k)$ be the integer corresponding to $a,T$ and $k$ in Table~\ref{table:q_a(T,k)}. Then:
				\[\text{If $q \leq q_a(T,k)$ and $a_\chi=a$, then $N(T,\chi)\leq k$.}\]
			Moreover, if $q_a(T,k)$ is one of the boldface entries of Table~\ref{table:q_a(T,k)}, then there is a character $\tau$ with conductor $q_a(T,k)+1$, $a_\chi = a_\tau$, and $N(T,\tau)>k$. 
	\end{enumerate}
\end{lemma}

\begin{proof}
Lemma~\ref{comp results}(a) is a partial verification of the generalized Riemann hypothesis. Although Platt~\cite{Platt} has made similar computations to much greater height with many more conductors, we independently confirm GRH to this level and make our rigorous zeros openly available at
\begin{center}
\texttt{\href{http://www.nt.math.ubc.ca/BeMaObRe2/}{\url{http://www.nt.math.ubc.ca/BeMaObRe2/}}}
\end{center}

The conditions in Lemma~\ref{comp results}(b) imply that $q\le 15$. The 40 primitive characters with $q\le 15$ are covered in our dataset, and for each, the lowest-height zero is excluded by $\log\frac{q(T+2)}{2\pi} \le 1.567$. 

Lemma~\ref{comp results}(c) requires many cases. For each of the 80818 relevant characters, the zeros are known to within $10^{-12}$. Between each pair of consecutive zeros $u,v$ (or before the first zero), we know the value of $N(T,\chi)$. This gives a range of $T$ over which the inequalities can be proved by our Moore--Skelboe-style interval arithmetic algorithm.
  
Lemma~\ref{comp results}(d), concerning the boldface and asterisked entries in Table~\ref{table:q_a(T,k)}, is also straightforward to pull from our dataset.
\end{proof}

The other entries in Table~\ref{table:q_a(T,k)} can be verified as follows.  For a given $a,T,k$, we find the values of $c$ and $r$ from Table~\ref{table:c and r2}. We then use equation~\eqref{eq:von Mangold}, evaluating the first two terms to many digits. The last term of equation~\eqref{eq:von Mangold} is bounded using Proposition~\ref{C2 prop} and Proposition~\ref{prop:Backlund2}. To use Proposition~\ref{prop:Backlund2}, we need to confirm that the restrictive inequalities hypothesized there are satisfied. Both~$E_\delta$ and~$E_{\sigma_1}$ can be explicitly computed, leaving only~$S$. In Proposition~\ref{prop:Jensen}, the quantity~$S$ is bounded in terms of the Jensen integral. The integrand in the Jensen integral is bounded in Lemma~\ref{lem: f bound RHP}, and the bound is restated in Definition~\ref{def:F}. The resulting integral is then rigorously bounded above using interval arithmetic, subdividing the region of integration until trivial bounds give the needed precision. Finally, as $N(T,\chi)$ must be an integer, we take a floor.

\begin{example}
For example, to verify the statement
    \[\text{if the conductor of $\chi$ is at most $25252$ and $a_\chi=0$, then $N(1,\chi)\leq 7$},\]
we take $T=1,q=25252,a=0$ and $c=\frac{2694}{2048}\approx 1.315$, $r=\frac{4651}{2048}\approx 2.271 $, with values of $c$ and $r$ being pulled from Table~\ref{table:c and r2}. Looking ahead to Proposition~\ref{prop:Backlund2}, we set 
$$
\sigma_1=\frac{1}{2}+\frac{835}{512 \sqrt{2}} \approx 1.653
$$
 and 
 $$
 \delta= \frac{835}{512}-\frac{835}{512 \sqrt{2}}\approx 0.478. 
 $$
 We find that
  \[
  \frac{T}{\pi} \log\frac{q}{\pi} + \frac{2}{\pi} \Im \ln\Gamma(\tfrac 14 + \tfrac {a}2+ i\, \tfrac T2) \leq 2.1013434, \; \; \frac 1\pi \cdot 2\log\zeta(\sigma_1) \leq 0.4883702
  \]
  and
  \[
  E_\delta \leq 0.1616976, \; E_{\sigma_1} -E_\delta \leq  0.5119502, \;  \log \tfrac{r}{c-1/2} =\log \left( \frac{4651}{1670} \right).
   \]
At this point, as $r>(1+\sqrt 2)(c-\frac 12)$, we can appeal to  Proposition~\ref{prop:Backlund2} to find that
  \begin{equation*}
  N(T,\chi) \leq 2.1013434+ 0.4883702 + \frac 2\pi \bigg(  \frac{\pi \, S}{2\log \frac{4651}{1670} } + \frac {0.1616976}2 +\frac{0.5119502}{2} \bigg(1-\frac{\log(1+\sqrt	2)}{\log \frac{4651}{1670}}\bigg) \bigg),
  \end{equation*}  
  whereby we may conclude that
  $$
   N(T,\chi) \leq  2.6639165+0.9763160 \cdot S.
    $$
From Proposition~\ref{prop:Jensen}, 
  \[ S \leq \log\frac{\zeta(c)}{\zeta(2c)} + \frac 1\pi \int_0^\pi F_{c,r}(\theta)\,d\theta,\]
with $F_{c,r}(\theta)$ made explicit in Definition~\ref{def:F}.
Easily computing $\log\frac{\zeta(c)}{\zeta(2c)} \leq 1.0682664$, and using interval arithmetic branch-and-bound, we find that 
  \[\int_0^\pi F_{c,r}(\theta)\,d\theta \leq 13.8132592.\]
Thus,
  \[N(T,\chi) \leq 2.663915+ 0.9763160 \left(1.0682664 +\frac{13.8132592}{\pi }\right) \leq 7.9997.\]
As $N(T,\chi)$ must be an integer, necessarily $N(T,\chi) \leq 7$.
\end{example}

For the entries in Table~\ref{table:q_a(T,k)} with asterisks, the method just described yields bounds that are inferior to the results of our brute-force computations recorded in Lemma~\ref{comp results}, and so the entries that appear are taken from those computations instead of the theoretical bound.

\begin{table}
$$
\begin{array}{c|cc|cc|cc}
  & \multicolumn{2}{|c|}{T= 5/7} & \multicolumn{2}{|c|}{T= 1} & \multicolumn{2}{|c}{T= 2} \\ \hline
k & a=0 & a=1 & a=0 & a=1 & a=0 & a=1 \\ \hline
0 & \bf{42} & \bf{16} & \bf{36} & \bf{12} & \bf{16} & \bf{10} \\
1 &  \bf{172} & \bf{66} & \bf{148} & \bf{42} & \bf{28} & \bf{18} \\
2 & 934^{\star} & 934^{\star} & 844^{\star} &  \bf{408} & \bf{120} & \bf{64} \\
3 & 934^{\star} & 934^{\star} & 844^{\star} & 844^{\star} &  \bf{330} & \bf{210} \\
4 & 934^{\star} & 934^{\star} & 844^{\star} & 844^{\star} & 634^{\star} &  \bf{630} \\
5 & 3289 & 1909 & 1616 & 905 & 634^{\star} & 634^{\star} \\
6 & 15991 & 9007 & 6256 & 3425 & 660 &  634^{\star} \\
7 & 82233 & 45137 & 25252 & 13554 & 1669 & 1050 \\
8 & 443412  & 238003  & 105597 & 55727 &  4289 & 2677  \\
9 & 2489523 & 1310445 &  455195 & 236710 & 11185 & 6932 \\ \hline
\end{array}
$$
\caption{If $\chi$ has sign $a$ and conductor $q\leq q_a(T,k)$, then $N(T,\chi) \leq k$. For example, if $\chi$ has sign $a=1$ and conductor less than  9007, then $N(5/7,\chi) \leq 6$.   The numbers in boldface are  best possible.}
\label{table:q_a(T,k)}
\end{table}

\begin{table}
$$
\begin{array}{c|cc|cc|cc}
  & \multicolumn{2}{|c|}{T= 5/7} & \multicolumn{2}{|c|}{T= 1} & \multicolumn{2}{|c}{T= 2} \\ \hline
k & a=0 & a=1 & a=0 & a=1 & a=0 & a=1 \\ \hline
5 & (2822,5006) & (2896,5176) & (2886,5212) & (2961,5388) &  &  \\
6 & (2719,4694) & (2770,4836) & (2778,4902) & (2831,5046) & (2956,5481) &  \\
7 & (2640,4447) & (2677,4566)  & (2694,4651) & (2734,4771) & (2861,5221) & (2906,5346) \\
8 & (2577,4246)  & (2606,4348)  & (2628,4444)  & (2660,4546) &  (2785,5001) & (2822,5107)  \\
9 & (2527,4081) & (2550,4168) &  (2575,4272) & (2600,4358) & (2723,4812) & (2753,4904) \\ \hline
\end{array}
$$
\caption{Values of pairs $(c^{\star},r^{\star})$ where  $c=c^{\star}/2^{11}$ and $r=r^{\star}/2^{11}$ that can be used to justify the entries in Table~\ref{table:q_a(T,k)}.}
	\label{table:c and r2}

\end{table}

\section*{Acknowledgements}

The first, second, and fourth authors were supported by NSERC Discovery Grants. Support for this project was provided to the third author by a PSC-CUNY Award, jointly funded by The Professional Staff Congress and The City University of New York.


\begin{bibdiv}
  \begin{biblist}
  	      \bib{Alirezaei}{article}{
  		author={Alirezaei, G.},
  		title={A Sharp Double Inequality for the Inverse Tangent Function},
  		eprint={arXiv:1307.4983},
  	}
  	
 \bib{EBPAP}{article}{
   author={Bennett, M. A.},
   author={Martin, G.},
   author={O'Bryant, K.},
   author={Rechnitzer, A.},
   title={Explicit bounds for primes in arithmetic progressions},
   journal={Illinois J. Math.},
   volume={62},
   date={2018},
   number={1-4},
   pages={427--532},
   issn={0019-2082},
   review={\MR{3922423}},
   doi={10.1215/ijm/1552442669},
 }
 \bib{Brent}{article}{
   author={Brent, R. P.},
   title={On asymptotic approximations to the log-Gamma and Riemann-Siegel theta functions},
   journal={J. Aust. Math, Soc.}
   date={December 21, 2018},
   eprint={arXiv:1609.03682v2},
   doi={10.1017/S1446788718000393}
   }
\bib{Hare}{article}{
   author={Hare, D. E. G.},
   title={Computing the principal branch of log-Gamma},
   journal={J. Algorithms},
   volume={25},
   date={1997},
   number={2},
   pages={221--236},
   issn={0196-6774},
   review={\MR{1478568}},
   doi={10.1006/jagm.1997.0881},
 }
 \bib{McCurley}{article}{
   author={McCurley, K. S.},
   title={Explicit estimates for the error term in the prime number theorem
   for arithmetic progressions},
   journal={Math. Comp.},
   volume={42},
   date={1984},
   number={165},
   pages={265--285},
   issn={0025-5718},
   review={\MR{726004}},
   doi={10.2307/2007579},
 }
 \bib{PARI2}{article}{
    author={The PARI~Group}, 
    title={PARI/GP version {\tt 2.11.0}}, 
    journal={Univ. Bordeaux},
    date={2018},
    eprint={http://pari.math.u-bordeaux.fr/}
    }
    
    \bib{Platt}{article}{
    author={Platt, D. J.},
   title={Numerical computations concerning the GRH},
   journal={Math. Comp.},
   volume={85},
   date={2016},
   pages={3009--3027},
   doi={10.1090/mcom/3077},
 }

 \bib{Rademacher}{article}{
   author={Rademacher, H.},
   title={On the Phragm\'{e}n-Lindel\"{o}f theorem and some applications},
   journal={Math. Z},
   volume={72},
   date={1959/1960},
   pages={192--204},
   issn={0025-5874},
   review={\MR{0117200}},
   doi={10.1007/BF01162949},
 }
 \bib{Selberg}{article}{
   author={Selberg, A.},
   title={Contributions to the theory of Dirichlet's $L$-functions},
   journal={Skr. Norske Vid. Akad. Oslo. I.},
   volume={1946},
   date={1946},
   number={3},
   pages={62},
   review={\MR{0022872}},
}
\bib{Topsoe}{article}{
	author={Tops\o e, F.},
	title={Some bounds for the logarithmic function},
	conference={
		title={Inequality theory and applications. Vol. 4},
	},
	book={
		publisher={Nova Sci. Publ., New York},
	},
	date={2007},
	pages={137--151},
	review={\MR{2349596}},
}

\bib{Trudgian0}{article}{
   author={Trudgian, T. S.},
   title={An improved upper bound for the argument of the Riemann zeta-function on the critical line II},
   journal={J. Number Theory},
   volume={134},
   date={2014},
   pages={280--292},
 }
\bib{Trudgian}{article}{
   author={Trudgian, T. S.},
   title={An improved upper bound for the error in the zero-counting
   formulae for Dirichlet $L$-functions and Dedekind zeta-functions},
   journal={Math. Comp.},
   volume={84},
   date={2015},
   number={293},
   pages={1439--1450},
   issn={0025-5718},
   review={\MR{3315515}},
   doi={10.1090/S0025-5718-2014-02898-6},
 }
\end{biblist}
\end{bibdiv}
\end{document}